\documentclass[12pt]{amsart}

\usepackage{amssymb}
\usepackage{amsmath}
\usepackage{graphicx}
\usepackage{epstopdf}
\usepackage{color}
\usepackage{overpic}

\newcommand{\Z}{\mathbb{Z}}
\newcommand{\R}{\mathbb{R}}

\def\dash{\discretionary{-}{}{-}\penalty1000\hskip0pt}

\newcommand{\spann}{{\rm span}}
\newcommand{\wu}{{\sf WU}}
\newcommand{\crr}{{\rm cr}}

\theoremstyle{definition}
\newtheorem{theorem}{Theorem}
\newtheorem{definition}[theorem]{Definition}
\newtheorem{proposition}[theorem]{Proposition}
\newtheorem{remark}[theorem]{Remark}

\newtheorem{corollary}[theorem]{Corollary}
\newtheorem{example}[theorem]{Example}

\makeatletter\def\@captionfont{\normalfont\footnotesize}\makeatother

\begin{document}

\title[Realizability of tensegrities in higher dimensions]{Geometric criteria for realizability of tensegrities in higher dimensions}

\author{Oleg Karpenkov \and Christian M\"uller}

\address{Oleg Karpenkov\\
University of Liverpool
}
\email{karpenk@liv.ac.uk}

\address{Christian M\"uller\\
TU Wien
}
\email{cmueller@geometrie.tuwien.ac.at}

\keywords{Tensegrity, self-stressed equilibrium framework, join/intersection operations, projective geometry, multiplicative $1$-forms}


\maketitle

\begin{abstract}
  In this paper we study a classical Maxwell question on the existence of
  self-stresses for frameworks, which are called tensegrities.  We give a
  complete answer on geometric conditions of at most $(d+1)$-valent
  tensegrities in $d$-dimensional space both in terms of discrete
  multiplicative 1-forms and in terms of ``join'' and ``intersection''
  operations in projective geometry.
\end{abstract}

\section{Introduction}

In this paper we are dealing with a classical question on self-stresses of frameworks in arbitrary dimensions (that were later referred as tensegrities).
Our main goal is to find geometric tensegrity existence characterizations
on all (generic) $k$-valent graphs in $\R^d$ ($k \leq d + 1$).
We do this in two different geometric settings.
The first one is based on discrete multiplicative $1$-forms which belong
to discrete differential geometry. For example, discrete multiplicative
$1$-forms have been used to characterize discrete Koenigs
nets~\cite{bobenko+2008}.
The second one is via geometric relations written in terms of
join/intersection operations in projective geometry.

\emph{Tensegrities} (Definition~\ref{defR}) are frameworks in equilibrium
consisting of rods, cables, and struts linked to each other at vertices $p_i$.
Forces in the framework (represented by vectors $w_{i, j} (p_i - p_j)$)
acting along cables are pulling $w_{i, j} < 0$ whereas forces acting along
rods are pushing $w_{i, j} > 0$; for the struts both force action is possible. The equilibrium condition means that all
forces around each vertex sum up to zero:
$$
  \sum\nolimits_{\{j|j\ne i\}} w_{i,j}(p_i-p_j)=0,
  \quad\text{for all}\
  i.
$$

Although the research on tensegrities was initiated already in 1864 by J.C.~Maxwell~\cite{Maxwell1864},
the term ``tensegrity'' itself appears much later.
Tensegrity is a concatenation of the words ``tension'' and ``integrity''.
This term was proposed by R.~Buckminster Fuller who
was inspired by the elegance of self-stressed constructions.
Tensegrities form an essential part of modern architecture and in arts, they serve as a light structural support
(like in a recent sculpture {\it TensegriTree} in the University of Kent).
Tensegrities are traditionally used in the study of cells~\cite{Ingber2014,Aloui2018},
viruses~\cite{Caspar2013,Simona-Mariana2011}, deployable mechanisms~\cite{Skelton1997}, etc.

In the second half of the 20th century the subject of tensegrities
became popular in mathematics again: questions of rigidity and flexibility
of structures were studied amongst others by R.~Connelly, B.~Roth, and W.~Whiteley in~\cite{Connelly2013,Connelly1996,Roth1981,Whiteley1997}, etc.
For a general modern overview of the subject we refer to the book~\cite{Sitharam2019}.

Tensegrities were generalized to spherical and projective geometries (by F.V.~Saliola and W.~Whiteley~\cite{Saliola2007});
to normed spaces (by D.~Kitson and S.C.~Power in~\cite{Kitson2014}
and by D.~Kitson and B.~Schulze in~\cite{Kitson2015});
and to surfaces in $\R^3$ (by B.~Jackson and A.~Nixon in~\cite{Jackson2015a}); etc.
\\

\noindent
{\bf Realizability of tensegrities.}
If the amount of edges is not large enough, a generic realization of a graph in $\R^d$ will not have a non-zero tensegrity.
The non-zero tensegrities exist only for specific frameworks (that are actually semi-algebraic sets in the configuration spaces of
tensegrities~\cite{Doray2010}).
For instance, a framework for the $K_{3,3}$ graph admits a non-zero tensegrity if and only if all its six points are on a conic.

An algebraic description of realizability conditions for tensegrities was proposed by N.L.~White and W.~Whiteley in~\cite{White1983,White1987}.
It was given in terms of bracket rings for the determinants of extended rigidity matrices (see also~\cite{White1975}).
This original algebraic approach introduces large polynomial
conditions using all vertices of the whole graph (plus several additional
vertices taken arbitrarily, they are called {\it tie-downs}) which can become rather complex and tricky to observe and to
analyze. Factorization of these polynomials is a hard problem that remains open since the seminal papers of N.L.~White and W.~Whiteley~\cite{White1983,White1987}.
This problem is of high importance in the area as the factors correspond to self-stressed frameworks
sharing certain geometric properties (like certain vertices in a line or certain planes intersect, etc).
The hope is that the factorization of these polynomials can be deduced from their rigidity nature (that was confirmed by several examples).

Our geometric approach is designed to complement this
algebraic approach.  It is more localised, here we write single conditions for
cycles in the graph which delivers explicit conditions on the
arrangements of affine spaces associated to frameworks.
Often the last provides factors of the bracket expression.
Here we would like to note that the amount of ``cycle'' conditions to be
considered simultaneously is suggested by the combinatorial theory of
tensegrities studied by R.~Connelly, B.~Roth, W.~Whiteley etc.\ (e.g., for
Laman graphs in the two-dimensional plane all conditions should be
one\dash dimensional).
A good reference for the combinatorial theory is the book~\cite{Sitharam2019}.

In their work M.~de~Guzm\'an and D.~Orden~\cite{Guzman2006} made first steps in the study of geometry of stresses by introducing atom decomposition techniques.
In all the studied examples (see, e.g.,~\cite{Doray2010,White1987})
there is  a simple geometric description for tensegrities
in terms of the ``meet-join'' operations of Cayley algebra.
This suggests such a description for all possible graphs.
In this paper we develop techniques to write such conditions for
the case of $k$-valent graphs ($k \leq d + 1$) in an arbitrary dimension $d$.

A preliminary investigation of geometric conditions was made in~\cite{Doray2010}: the authors had introduced two surgeries that result in classification of all the geometric conditions for codimension one strata for graphs with
8 or less vertices.
Topological properties of the configuration spaces of all tensegrities for graphs with 4 and 5 vertices were studied in~\cite{Karpenkov2013}.
A complete description of geometric conditions in the two-dimensional case was announced in~\cite{Karpenkov2019}.
Finally, a nice collection of problems on geometry and topology of
stratification of tensegrities can be found in~\cite{Karpenkov2018}.

In the present paper we consider less than $d{+}1$ valent graphs in $\R^d$.
We write geometric conditions both in terms of integrability of
multiplicative 1-forms (Theorem~\ref{MainTheorem1}) and in terms of
join/intersection operations in projective geometry
(Theorem~\ref{GC-conditions}).

\noindent
{\bf Organization of the paper.}
We start in Section~\ref{sec:notions} with the definition of
tensegrities and notions that we use throughout the paper.
In Section~\ref{sec:ratios} we discuss discrete multiplicative $1$-forms
and how exact $1$-forms characterize frameworks admitting non-zero
self-stresses.
In Section~\ref{sec:gca} we work within join/intersection operations
in projective geometry to provide a recursive geometric characterization
of tensegrities.
Section~\ref{sec:applications} is devoted to point out a relation between
tensegrities and harmonic maps.
Finally in Section~\ref{sec:octahedron}
we study examples of tensegrities in $\R^3$.

\section{Notions and definitions}
\label{sec:notions}

In this section we give the necessary definitions of the setting around
tensegrities. Additionally, we provide the notion of general position of
the framework so that we can formulate our geometric conditions on
frameworks admitting a tensegrity.

\subsection{Definition of tensegrities}

Let us first set the scene by recalling some basic notions before we come
to the general definition of tensegrities.

\begin{definition}\label{defR}
  Let $G$ be an arbitrary graph without loops and multiple edges on $n$
  vertices.

--- Let $V(G)=\{v_1, \ldots, v_n\}$ and $E(G)$ denote the sets of vertices and edges for $G$, respectively.
Denote by $(v_i;v_j)$ the edge joining $v_i$ and $v_j$.

  --- Let $B(G)$ be the subset of all 1-valent vertices in $V(G)$, which we refer to as the \emph{boundary of $G$}.

--- Let $Z(G)$ be the subset of all vertices with valence greater than 1 in $V(G)$.

\begin{itemize}
\item A \emph{framework} $G(P)$ is a map of the vertices $v_1, \ldots, v_n$ of $G$
onto a finite point configuration $P=(p_1,\ldots,p_n)$ in $\R^d$,
such that $G(P)(v_i)=p_i$ for $i=1,\ldots,n$.
We say that there is an \emph{edge} between $p_i$ and $p_j$ if $(v_i;v_j)$ is an edge of $G$
and denote it by $(p_i;p_j)$.
Note that the points $p_1,\ldots, p_n$ are not necessarily distinct.

\item A \emph{stress} $w$ on a framework is an assignment of real scalars $w_{i,j}$
(called \emph{tensions}) to its edges $(v_i;v_j)$ with the property $w_{i,j} = w_{j,i}$.
We also set $w_{i,j}=0$ if there is no edge between the corresponding vertices.

\item A stress $w$ is called a \emph{self-stress} if the
following equilibrium condition is fulfilled at every vertex of valence greater than 1, i.e., for all $v_i \in Z(G)$:
$$
\sum\limits_{\{j|j\ne i\}} w_{i,j}(p_i-p_j)=0.
$$
By $p_i-p_j$ we denote the vector from the point $p_j$ to the
point $p_i$.
    Note that we do not consider equilibrium for the boundary points $B(G)$.
These are the points where the framework is attached to the exterior construction.
Therefore, the corresponding forces are compensated by the forces of the exterior construction.

\item A pair $(G(P),w)$ is called a \emph{tensegrity} if
$w$ is a self-stress  for the framework $G(P)$.

\item A tensegrity $(G(P),w)$ (or stress $w$) is said to be \emph{non-zero}
if there exists an edge $(v_i;v_j)$ of the framework that
has non-vanishing tension $w_{i,j}\ne 0$.

\item A tensegrity $(G(P),w)$ (or stress $w$) is said to be
  \emph{everywhere non-zero} if each existing edge $(v_i;v_j)$ of the
    framework has non-vanishing tension $w_{i,j}\ne 0$.
\end{itemize}
\end{definition}

\begin{remark}
If the set of boundary points $B(G)$ is empty, we have the classical case
  of tensegrities without boundary.
\end{remark}

\subsection{Frameworks in various general positions}

To formulate our geometric conditions on frameworks admitting a non-zero
self-stress via discrete multiplicative $1$-forms, we need the vertices to
lie in general position (Sec.~\ref{sec:ratios}).
A slightly stronger version of generality is needed to formulate our
conditions within the setting of join/intersection operations in projective
geometry (Sec.~\ref{sec:gca}).

As a general notion throughout the paper, by $\spann(s_1,\ldots, s_k)$ we
denote the {\it affine} or {\it projective span}  of affine/projective
spaces $s_1,\ldots, s_k$ and not the linear span as a vector space.
By \emph{$2$-plane} we denote a two-dimensional affine/projective subspace
and by $k$-plane a $k$-dimensional affine/projective subspace.

\begin{definition}
  \label{defn:lingen}
  A framework $G(P)$ is \emph{linearly generic} if for every vertex (whose
  degree or valence we denote by $k$) the following two conditions hold:
\begin{itemize}
\item the $k$ edges emanating from this vertex span a $(k-1)$-plane;
\item every subset of $k-1$ edges emanating from this vertex spans this
  $(k-1)$-plane.
\end{itemize}
\end{definition}

\begin{remark}
The valences of a linearly generic framework in $\R^d$ do not exceed $d+1$.
\end{remark}

For the geometric characterization of tensegrities in terms of discrete
multiplicative $1$-forms (Section~\ref{sec:ratios}), the property on
frameworks of being linearly generic is all we need. As for our
characterization as formulated within join/intersection operations in
projective geometry (Section~\ref{sec:gca}), we have to include one
further notion of general position.

\begin{definition}
  \label{def:genpos}
  A framework $G(P) \subset \R^d$ is \emph{in 3D-general position} if the following two
  conditions hold:
\begin{itemize}
\item $G(P)$ is linearly generic, and
\item every 4-tuple of vertices in every cycle of $G(P)$ spans a $3$-plane.
\end{itemize}
\end{definition}

Note that the second condition in the previous definition does not
imply that the framework must lie in $\R^3$. Just every 4-tuple of
vertices in a cycle span a $3$-plane.

\section{Characterizing $k$-valent tensegrities in terms of ratios}
\label{sec:ratios}

In this section we give a geometric characterization for linearly generic
at most $k$-valent graphs in $\R^d$ ($k \leq d + 1$) admitting a non-zero
self-stress.
It turns out to be practical to first provide a characterization for
trivalent graphs before then generalizing it to $k$-valent graphs.
Throughout this section all graphs $G$ are connected.
Our goal is to show that the product of certain ratios is $1$ if and only
if the framework admits a non-zero tensegrity
(Theorem~\ref{MainTheorem1}).

\subsection{Tensegrities over trivalent graphs}
\label{subsec:trivalent}

Our geometric characterization of a linearly generic framework to be a
non-zero tensegrity is defined on the cycles of the underlying graph.
The important notion here is the one of a discrete multiplicative $1$-form
which is well known in discrete differential geometry. We follow the
definition in~\cite{bobenko+2008}.
\begin{definition}
  A real valued function $q : \vec E(G) \to \R\setminus\{0\}$
  ($\vec E(G)$ denotes the set of oriented edges of the graph $G$) is
  called a \emph{multiplicative $1$-form}, if $q(-e) = 1/q(e)$ for every
  $e \in \vec E(G)$.  It is called \emph{exact} if for every cycle $e_1,
  \ldots, e_k$ of directed edges the values of the $1$-form multiply to
  $1$, i.e.,
  $$
  q(e_1) \cdot \ldots \cdot q(e_k) = 1.
  $$
\end{definition}

The following definition is about a particular subdivision of a graph.

\begin{definition}
  The \emph{line graph} $L(G)$ of a general graph $G$ has the following
  properties. Its vertices are in a one-to-one correspondence with the
  edges of $G$. The edges of $L(G)$ connect two vertices if and only if
  the two respective edges of $G$ are emanating from the same vertex.  See
  Fig.~\ref{fig:midedgegraph} (left) and Fig.~\ref{fig:midedgegraphk}
  (left).
\end{definition}

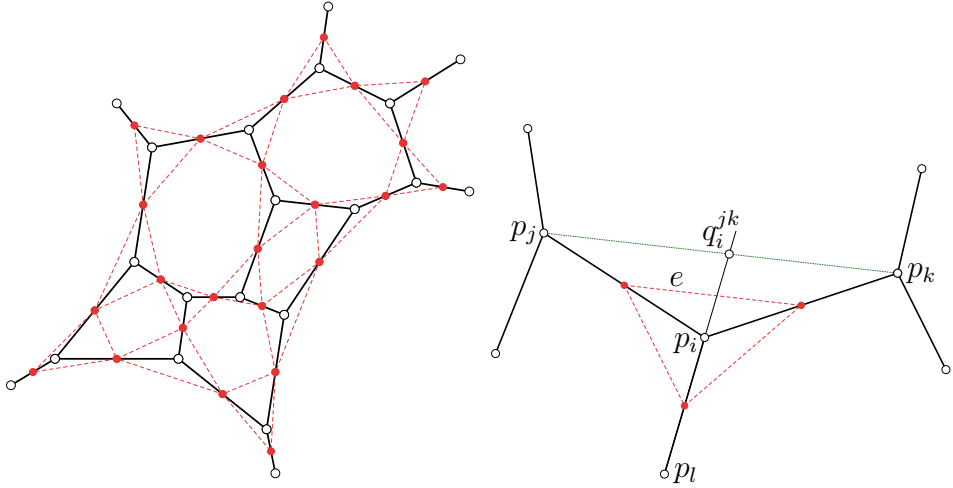
\begin{figure}[t]
  \begin{overpic}[width=.5\textwidth]{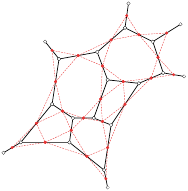}
  \end{overpic}
  \hfill
  \begin{overpic}[width=.49\textwidth]{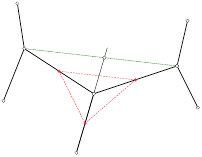}
    \put(5,53){$p_j$}
    \put(40,29){$p_i$}
    \put(90,44){$p_k$}
    \put(40,1){$p_l$}
    \put(46,52){$q_i^{jk}$}
    \put(39,42){$e$}
  \end{overpic}
  \caption{\emph{Left}: A graph $G$ (black lines; white vertices) and its
  line graph $L(G)$ (red dashed edges; red vertices).
  \emph{Right}: A flat vertex star $p_i, p_j, p_k, p_l$. The edge $e$ of the
  line graph $L(G)$ corresponds to the angle $(v_j, v_i, v_k)$.
  The intersection point $q_i^{jk}$ of the straight lines through $p_j
  p_k$ and $p_i p_l$ determines the value of the multiplicative $1$-form
  $q$ on that edge by $q(v_j, v_i, v_k) = (p_j - q_i^{jk}) : (q_i^{jk} -
  p_k)$.
  }
  \label{fig:midedgegraph}
\end{figure}

We now aim at constructing a multiplicative $1$-form on the oriented edges
of the line graph $L(G)$ of a trivalent graph $G$ corresponding
to a linearly generic framework $G(P)$.

Each edge $e$ of $L(G)$ connects the midpoints of edges of the form $(v_j;
v_i)$ and $(v_i; v_k)$, as illustrated in Figure~\ref{fig:midedgegraph}
(right).
We can therefore denote the oriented edges of $L(G)$ by triplets of the
form $e = (v_j, v_i, v_k)$ with the property that the negatively oriented
edge is $-e = (v_k, v_i, v_j)$.

Let us denote the third edge emanating from $v_i$ by $(v_i; v_l)$.
The framework being linearly generic implies that the corresponding vertices
$p_i, p_j, p_k, p_l$ lie in a common $2$-plane. Furthermore, the framework
being linearly generic implies that the straight line connecting $p_i p_l$
intersects the line connecting $p_j p_k$ in a point $q_i^{jk}$.
Consequently, this point gives rise to an affine ratio of the form
\begin{equation}\label{eq:q}
  q(v_j, v_i, v_k) := \frac{p_j - q_i^{jk}}{q_i^{jk} - p_k},
\end{equation}
as ratio of parallel vectors. Clearly, $q(v_j, v_i, v_k) = 1/q(v_k, v_i,
v_j)$ which implies that $q$ is a multiplicative $1$-form on the oriented
edges of the line graph $L(G)$.

\begin{theorem}\label{thm:ratios}
  Let $G(P)$ be a linearly generic trivalent framework. Then there is
  a stress $w$ on $G(P)$ such that the framework $(G(P), w)$ is a non-zero
  tensegrity if and only if the $1$-form $q$ given by Equation~\eqref{eq:q} on
  the line graph $L(G)$ is exact.
\end{theorem}
\begin{proof}
  Let us first assume that $(G(P), w)$ is a non-zero tensegrity. Since at
  every inner vertex $p_i$ of a trivalent tensegrity the sum of forces
  adds up to zero we obtain
  \begin{equation}\label{eq:star}
    w_{i, j} (p_i - p_j) +
    w_{i, k} (p_i - p_k) +
    w_{i, l} (p_i - p_l) = 0.
  \end{equation}
  The point $q_i^{jk}$ lies on the straight line through $p_i p_l$ and can
  therefore be written in the form
  $$
  q_i^{jk} = p_i + \lambda (p_i - p_l)
  $$
  for some $\lambda \in \R$. Inserting Equation~\eqref{eq:star} yields
  \begin{align*}
    q_i^{jk}
    &=
    p_i + \lambda
    \Big(
    \frac{w_{i, j}}{w_{i, l}} (p_j - p_i)
    +
    \frac{w_{i, k}}{w_{i, l}} (p_k - p_i)
    \Big)
    \\
    &=
    \Big(
    1
    - \lambda \frac{w_{i, j}}{w_{i, l}}
    - \lambda \frac{w_{i, k}}{w_{i, l}}
    \Big)
    p_i
    + \lambda \frac{w_{i, j}}{w_{i, l}} p_j
    + \lambda \frac{w_{i, k}}{w_{i, l}} p_k.
  \end{align*}
  Since $q_i^{jk}$ must lie on the line through $p_j p_k$ we obtain for
  $\lambda = \frac{w_{i, l}}{w_{i, j} + w_{i, k}}$ and therefore the
  affine combination
  $$
  q_i^{jk} =
  \frac{w_{i, j}}{w_{i, j} + w_{i, k}} p_j
  +
  \frac{w_{i, k}}{w_{i, j} + w_{i, k}} p_k.
  $$
  Consequently, for our $q$ in Equation~\eqref{eq:q} we obtain
  \begin{equation}
    \label{eq:fracW}
  q(v_j, v_i, v_k) = \frac{w_{i, k}}{w_{i, j}}.
  \end{equation}

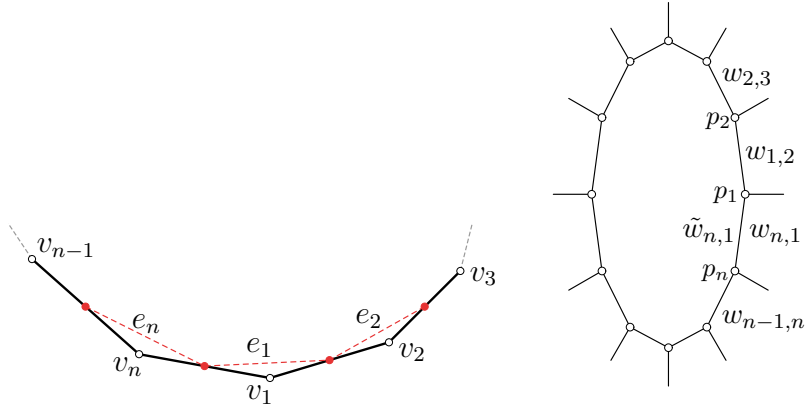
\begin{figure}[t]
  \hfill
  \begin{overpic}[width=.5\textwidth]{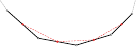}
    \put(6,29){$v_{n - 1}$}
    \put(22,4){$v_n$}
    \put(50,-2){$v_1$}
    \put(82,7){$v_2$}
    \put(97,23){$v_3$}
    \put(26,13){$e_n$}
    \put(50,8){$e_1$}
    \put(73,14){$e_2$}
  \end{overpic}
  \hfill
  \begin{overpic}[width=.25\textwidth]{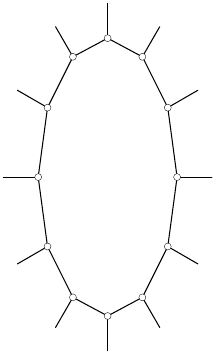}
    \put(34,39){\small$\tilde w_{n, 1}$}
    \put(44,17){\small$w_{n - 1, n}$}
    \put(51,39){\small$w_{n, 1}$}
    \put(50,59){\small$w_{1, 2}$}
    \put(44,79){\small$w_{2, 3}$}
    \put(40,68){\footnotesize$p_2$}
    \put(42,49){\footnotesize$p_1$}
    \put(39,29){\footnotesize$p_n$}
  \end{overpic}
  \hfill
  \hfill{}
  \caption{\emph{Left}: Notations of edges in a cycle in the line graph
  $L(G)$.
  \emph{Right}: A cycle in a trivalent graph. After prescribing a
  tension $w_{1, 2}$ we can compute the tension at edge $(v_n; v_1)$ in
  two ways: First, by enforcing equilibrium at $p_1$ (and getting $\tilde
  w_{n, 1}$), and second by transporting the tension along the cycle
  (resulting in $w_{n, 1}$).}
  \label{fig:cycle}
\end{figure}

  To show the exactness of $q$ we have to show that the product of all
  values along any cycle multiply to $1$. So let $(e_1, \ldots, e_n)$ be a
  cycle of the line graph $L(G)$ where $e_i$ are oriented edges.
  There is a corresponding cycle $(v_1, \ldots, v_n)$ in $G$ such that $e_i$
  corresponds to the angle $(v_{i - 1}, v_i, v_{i + 1})$, where we take
  the indices modulo $n$ (see Figure~\ref{fig:cycle} left).  We compute
  the product of corresponding values of $q$:
  \begin{align*}
    \prod_{i = 1}^n q(e_i)
    &=
    \prod_{i = 1}^n q(v_{i - 1}, v_i, v_{i + 1})
    =
    \prod_{i = 1}^n \frac{w_{i, i + 1}}{w_{i, i - 1}}
    =
    \prod_{i = 1}^n \frac{w_{i, i + 1}}{w_{i - 1, i}}
    \\
    &=
    \frac{w_{1, 2}}{w_{n, 1}}
    \cdot
    \frac{w_{2, 3}}{w_{1, 2}}
    \cdot
    \ldots
    \cdot
    \frac{w_{n - 1, n}}{w_{n - 2, n - 1}}
    \cdot
    \frac{w_{n, 1}}{w_{n - 1, n}}
    = 1,
  \end{align*}
  which shows the first direction of the statement.

  As for the other direction, let us first note that prescribing one
  tension $w_{i, j}$ in a trivalent vertex of a linearly generic
  framework uniquely determines the other two tensions as well since
  Equation~\eqref{eq:star} is then a linear combination of two linearly
  independent vectors with coefficients $w_{i, k}$ and $w_{i, l}$.
  Consequently, after choosing one tension $w_{i, j}$ we can transport it
  to any other vertex along any connected path. This way we could define a
  stress $w$ on the graph $G$, if this construction would be well-defined,
  i.e., if transporting the tension along different paths to the same
  edge would result in the same tensions. Or equivalently, if we
  transport the tension around any cycle we would have to get back to the
  same tension with which we started.

  So let us take an arbitrary cycle $(v_1, \ldots, v_n)$. We choose a non-zero
  tension $w_{1, 2}$ on the edge $(v_1; v_2)$ which immediately determines
  the tension $\tilde w_{n, 1}$ on the edge $(v_n; v_1)$ due to the
  equilibrium condition shown in Equation~\eqref{eq:star}. See also
  Figure~\ref{fig:cycle} (right). The value of the multiplicative
  $1$-form on the oriented edge $(v_{n}, v_1, v_2)$ of $L(G)$ therefore
  has the value
  $$
  q(v_{n}, v_1, v_2) = w_{1, 2}/\tilde w_{n, 1}.
  $$

  On the other hand $w_{1, 2}$ determines the tension $w_{2, 3}$ as edge
  emanating from $v_1$.  Repeating this propagation process we define all
  tensions in the cycle including the last one $w_{n, 1}$. We have
  therefore defined the tension at $(v_n; v_1)$ twice: from the ``left''
  and from the ``right'' as $\tilde w_{n, 1}$ and $w_{n, 1}$. Now the question
  is whether those tensions are the same.

  Our assumption is that the multiplicative $1$-form $q$ is exact which
  implies
  \begin{align*}
    1
    =
    \prod_{i = 1}^n q(v_{i - 1}, v_i, v_{i + 1})
    =
    \frac{w_{1, 2}}{\tilde w_{n, 1}}
    \cdot
    \frac{w_{2, 3}}{w_{1, 2}}
    \cdot
    \ldots
    \cdot
    \frac{w_{n - 1, n}}{w_{n - 2, n - 1}}
    \cdot
    \frac{w_{n, 1}}{w_{n - 1, n}}
    =
    \frac{w_{n, 1}}{\tilde w_{n, 1}},
  \end{align*}
  and therefore $w_{n, 1} = \tilde w_{n, 1}$. Consequently, we can
  consistently define a stress $w$ (uniquely up to scaling) on $G(P)$ such
  that the framework $(G(P), w)$ is a non-zero tensegrity.
\end{proof}

The following corollary follows immediately from Theorem~\ref{thm:ratios}
and its proof, in particular from Equation~\eqref{eq:fracW}.

\begin{corollary}
  Let $G(P)$ be a linearly generic trivalent framework and let $w$ be
  a non-zero stress on $G(P)$. Then the framework $(G(P), w)$ is a non-zero
  tensegrity if and only if the $1$-form
  $$
    \tilde q(v_j, v_i, v_k) := \frac{w_{i, k}}{w_{i, j}},
  $$
  defined on the line graph $L(G)$ is exact.
  \qed
\end{corollary}

\subsection{Special cases of trivalent cycles}
\label{subsec:specialcase}

In this section we will consider two special cases of cycles and briefly
reflect on what Theorem~\ref{thm:ratios} means in these cases.

\underline{$n = 3$}: In that case the cycle is a triangle and the points
$q_i^{jk}$ lie on the edges of the triangle opposite to $p_i$.
Consequently, the exactness of the $1$-form on that cycle is precisely the
setting of the classical Ceva's theorem (see e.g.,~\cite{richter-gebert-2011}). Therefore, the three
lines $p_1 q_1^{2,3}$, $p_2 q_2^{3,1}$, and $p_3 q_3^{1,2}$ intersect in one point
(cf.~\cite{Karpenkov2019} and see Figure~\ref{fig:specialcase} left).

\underline{$n = 4$}: In the case of a quadrilateral the points $q_i^{jk}$
lie on the diagonals (see Figure~\ref{fig:specialcase} right). Exactness of the
$1$-form on that cycle is equivalent to
$$
1 =
q(v_4, v_1, v_2)
\cdot
q(v_1, v_2, v_3)
\cdot
q(v_2, v_3, v_4)
\cdot
q(v_3, v_4, v_1),
$$
which is further equivalent to
$$
q(v_4, v_1, v_2)
\cdot
q(v_2, v_3, v_4)
=
\frac{1}{
q(v_1, v_2, v_3)
\cdot
q(v_3, v_4, v_1)
},
$$
and further to
$$
\frac{p_4 - q_1^{4, 2}}{q_1^{4, 2} - p_2}
\cdot
\frac{p_2 - q_3^{2, 4}}{q_3^{2, 4} - p_4}
=
\frac{q_2^{1, 3} - p_3}{p_1 - q_2^{1, 3}}
\cdot
\frac{q_4^{3, 1} - p_1}{p_3 - q_4^{3, 1}}.
$$
The last equation is an equation of cross-ratios, namely
\begin{equation}
  \label{eq:cr}
\crr(q_1^{4, 2}, p_4, q_3^{2, 4}, p_2)
=
\crr(q_2^{1, 3}, p_3, q_4^{3, 1}, p_1).
\end{equation}

\begin{figure}[t]
  \hfill
  \begin{overpic}[width=.36\textwidth]{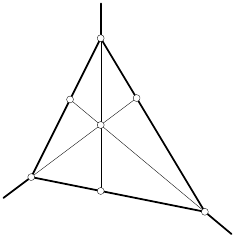}
    \put(12,18){$p_1$}
    \put(80,5){$p_2$}
    \put(46,85){$p_3$}
    \put(61,58){\small$q_1^{2,3}$}
    \put(18,57){\small$q_2^{3,1}$}
    \put(44,22){\small$q_3^{1,2}$}
  \end{overpic}
  \hfill
  \begin{overpic}[width=.36\textwidth]{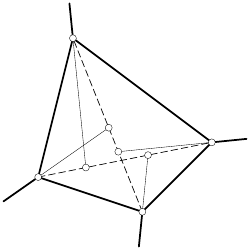}
    \put(12,22){$p_1$}
    \put(60,13){$p_2$}
    \put(84,47){$p_3$}
    \put(19,84){$p_4$}
    \put(44,52){\small$q_1$}
    \put(61,34){\small$q_2$}
    \put(48,42){\small$q_3$}
    \put(32,27){\small$q_4$}
  \end{overpic}
  \hfill{}
  \caption{\emph{Left}: The cycle is a triangle. The exactness of the
  $1$-form on that triangle is equivalent to the three ``outward''
  pointing edges intersecting in one point, i.e., Ceva's configuration.
  \emph{Right}: The cycle is a quadrilateral. Then the three ``outward''
  pointing edges intersect the respective diagonals in points $q_i$.  We
  abbreviate $q_1^{2, 4}$ simply by $q_1$ etc.  The exactness of the
  $1$-form is equivalent to
  $\crr(q_1, p_4, q_3, p_2)
  =
  \crr(q_2, p_3, q_4, p_1)$.
  }
  \label{fig:specialcase}
\end{figure}
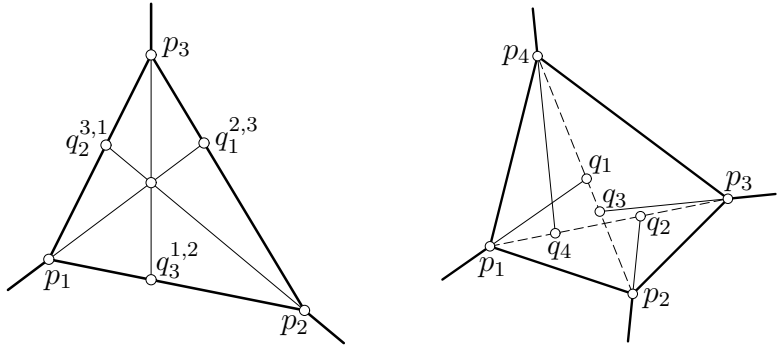

\begin{example}
  It is well known (see e.g., \cite{Doray2010}) that the complete graph
  $K_{3,3}$, which is trivalent, with vertices in $\R^2$ is a tensegrity
  if and only if the vertices lie on a conic (see Figure~\ref{fig:conic}).
  That property can also be shown easily within our setting of exact
  multiplicative $1$-forms as follows.
  According to Steiner's definition of conics, the property of six points
  lying on a conic is equivalent to the four lines $p_5 p_i$ and $p_6 p_i$
  for $i = 1, \ldots, 4$ being related by a projectivity (a projective
  map). Or equivalently that means that the cross-ratios of these pair of
  four lines are the same. Let us consider the cycle $(v_1, v_2, v_3,
  v_4)$ with four vertices. Then Equation~\eqref{eq:cr} holds for this
  cycle which we will use in the following computation. Further, we have
  \begin{align*}
    &\crr(p_5 p_2, p_5 p_3, p_5 p_4, p_5 p_1)
    =
    \crr(q_2, p_3, q_4, p_1)
    \overset{\eqref{eq:cr}}{=}
    \crr(q_1, p_4, q_3, p_2)
    \\
    =&
    \crr(p_6 p_1, p_6 p_4, p_6 p_3, p_6 p_2)
    =
    \crr(p_6 p_2, p_6 p_3, p_6 p_4, p_6 p_1),
  \end{align*}
  where the last equality holds because $\crr(a, b, c, d) = \crr(d, c, b, a)$.
\end{example}

\begin{figure}[h]
  \centerline{
  \begin{overpic}[width=.5\textwidth]{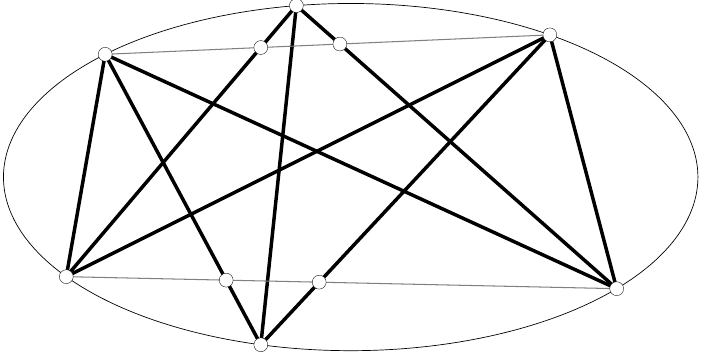}
    \put(10,45){\footnotesize$p_1$}
    \put(7,7){\footnotesize$p_2$}
    \put(80,47){\footnotesize$p_3$}
    \put(88,5){\footnotesize$p_4$}
    \put(37,52){\footnotesize$p_5$}
    \put(31,-2){\footnotesize$p_6$}
    \put(28,7){\footnotesize$q_1$}
    \put(32,45.4){\footnotesize$q_2$}
    \put(45,6){\footnotesize$q_3$}
    \put(50,46){\footnotesize$q_4$}
  \end{overpic}
}
  \caption{Six points $p_1, \ldots, p_6$ in the $2$-plane $\R^2$ form a
  tensegrity if and only if the six points lie on a conic.
  }
  \label{fig:conic}
\end{figure}

\begin{figure}[h]
  \begin{overpic}[width=.49\textwidth]{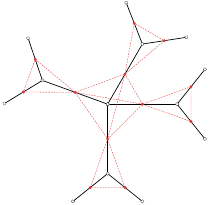}
  \end{overpic}
  \hfill
  \begin{overpic}[width=.49\textwidth]{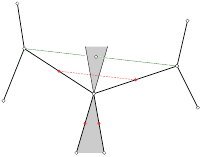}
    \put(5,53){$p_j$}
    \put(40,29){$p_i$}
    \put(90,44){$p_l$}
    \put(30,1){$p_n$}
    \put(54,1){$p_m$}
    \put(47,53){$q_i^{jl}$}
    \put(36,43){$e$}
  \end{overpic}
  \caption{\emph{Left}: The line graph $L(G)$ of a
  \emph{general} graph $G$.
  The new edges (red dashed) connect midpoints of old adjacent edges. A
  vertex star of valence three generates three new edges, a vertex star of
  valence four generates six new edges.
  \emph{Right}: We construct the discrete multiplicative $1$-form on edges
  of the line graph by intersecting the line $p_j p_l$
  with the affine subspace $\spann(\bigcup_{m \neq j, l} p_m)$.
  }
  \label{fig:midedgegraphk}
\end{figure}
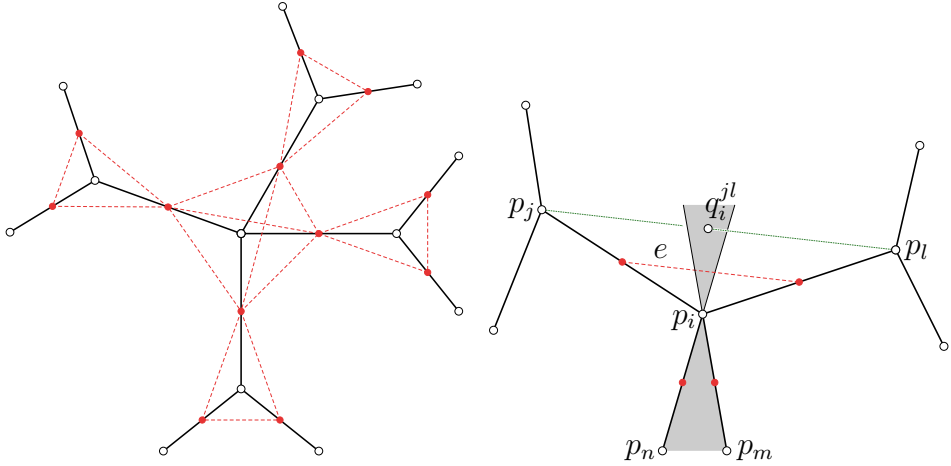

\subsection{Tensegrities over $k$-valent graphs}

We will now generalize the geometric characterization of trivalent
tensegrities (of Section~\ref{subsec:trivalent}) to linearly generic
$k$-valent tensegrities in $\R^d$ with $k \leq d + 1$.

Let us consider a $k$-valent vertex star with inner vertex $v_i$ and
adjacent vertices $v_1, \ldots, v_k$. Again we can denote an oriented
edge of the line graph $L(G)$ by $(v_j, v_i, v_l)$ (with $1 \leq
j \neq l \leq n$). See also Figure~\ref{fig:midedgegraphk} (right). Since
$G$ is linearly generic, the subspaces
$\spann(p_i, p_j, p_l)$
and
$\spann(\! \bigcup\limits_{m \neq j, l \atop v_m \sim v_i} p_m)$
intersect in a line $L$, where $v_m \sim v_i$ means $v_m$ is adjacent to
$v_i$. Consequently, this line $L$
intersects the line $p_j p_l$ in a point $q_i^{jl}$. In the trivalent
case, $L$ is simply the line $p_i p_l$.
Analogously to the trivalent case we define the discrete multiplicative
$1$-form as
\begin{equation}
  \label{eq:kq}
  q(v_j, v_i, v_l) := \frac{p_j - q_i^{jl}}{q_i^{jl} - p_l}.
\end{equation}
Now the proof of Theorem~\ref{thm:ratios} can be repeated basically word
by word which implies the following theorem.
\begin{theorem}\label{MainTheorem1}
  Let $G(P)$ be a linearly generic framework in $\R^d$ with vertices of
  valence at most $d + 1$. Then there is a stress $w$ on $G(P)$ such that
  the framework $(G(P), w)$ is a non-zero tensegrity if and only if the
  $1$-form $q$ given by Equation~\eqref{eq:kq} on the line graph $L(G)$ is
  exact.
  \qed
\end{theorem}

  \noindent
  \begin{minipage}{.79\textwidth}
Note that if a cycle can be decomposed by two other cycles with common
edges, the product of the multiplicative $1$-form values of the first
cycle equals the product of the decomposing cycles. The values of common
edges that appear in the decomposing cycles in reversed orientations
are reciprocal to each other and simply cancel out in the product. That
leads us to the following remark.
  \end{minipage}
  \hfill
  \begin{minipage}{.19\textwidth}
    \begin{overpic}[width=\textwidth]{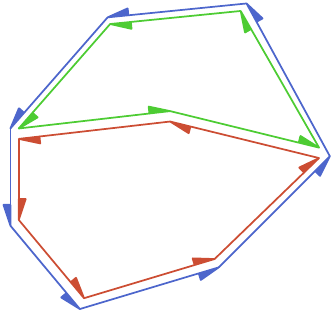}
    \end{overpic}
  \end{minipage}

\begin{remark}
  \label{rem:h1}
  In the previous theorem and in Theorem~\ref{thm:ratios} it is
  sufficient to check the criterion only for
  generator loops of the first homology group $H_1(G)$ of the graph (if we
  consider the graph as topological space), because all other loops can be
  decomposed by those.
  These conditions for different generators of  $H_1(G)$ may still
  coincide, as the realisability codimension in the space of all
  (including zero-force load) tensegrities is defined by the combinatorics
  of the graphs (e.g., it is one condition for Laman graphs in the
  plane).  Some of the conditions will correspond to different strata.
\end{remark}

\section{Join/intersection conditions for frameworks in 3D-general position}\label{sec:gca}

In this section we construct join/intersection conditions for frameworks
whose all 4-tuples of vertices in any cycle span a $3$-plane.
We start in Section~\ref{Framed cycles and their frameworks}
with the case of frameworks for so-called framed cycles.
We introduce $\wu$-surgeries on framed cycles that preserve the property to admit
a non-zero tensegrity and that reduce the amount of vertices of framed cycles.
These properties will lead to explicit expressions in terms of
join/intersection operations in projective geometry.
Further, in Section~\ref{On existence and uniqueness of tensegrities for
frameworks in general position}
we prove that a sufficiently generic framework admits a non-zero tensegrity
if and only if all its associated framed cycle frameworks
admit non-zero tensegrities (Theorem~\ref{cycle-tens}).
Finally, in Section~\ref{subsec:projcond} we briefly recall the basic
notions join/intersection in projective geometry and
construct conditions within that framework for the existence of tensegrities
for given graphs that are not generically flexible
(Theorem~\ref{GC-conditions}).
All frameworks in this section are in $\R^d$ with $d \geq 3$.

\subsection{Framed cycles and their frameworks}\label{Framed cycles and their frameworks}
We start this section with basic definitions, some properties of framed
cycles and their generic frameworks.
Further, we introduce $\wu$-surgeries that take
frameworks in 3D-general position to generic frameworks of framed cycles.
We show also that $\wu$-surgeries preserve the property of admitting a non-zero tensegrity.

\subsubsection{General definitions}

We say that a graph is a \emph{cycle} if it is homeomorphic to a circle.

\begin{definition}
Let $C=(c_1,\ldots,c_n)$ and $B=(b_{1},\ldots,b_{n})$ be two $n$-tuples of points.
  A \emph{framed cycle} $C_B=(C,B)$ is the cycle $c_1,\ldots,c_n$ with attached edges $b_ic_i$  for $i=1,\ldots, n$.
\end{definition}

\begin{definition}\label{def-3d-g-f}
  We say that a framework $C_B(P)$ of a framed cycle $C_B$ is \emph{in
  3D-general flat position} if
  \begin{itemize}
    \item $C_B(P)$ is linearly generic (see Definition~\ref{defn:lingen});
    \item there are no four points of $C(P)$ contained in a $2$-plane (only
      for the cycle $C$).
  \end{itemize}
\end{definition}

\begin{remark}
  Notice that linear genericity in particular implies that all edges
  emanating from the same vertex of a framed cycle are contained in a
  $2$-plane; and that $G(P)(b_i)\ne G(P)(c_i)$ for all admissible $i$.
\end{remark}

\subsubsection{A preliminary observation}

Let us formulate a preliminary statement for the definition of $\wu$-surgeries.

Recall that a cycle $C_B(P)$ is in 3D-general flat position if it is linearly generic (every three edges emanating from the same vertex span a $2$-plane)
and if there are no four points of $C_B(P)$ contained in a $2$-plane (see Definition~\ref{def-3d-g-f}).

\begin{proposition}\label{Predefinition-HF}
Let a framed cycle framework $C_B(P)$ be in 3D-general flat position.
Let also
$$
  G(P)(b_i)=e_i, \quad G(P)(c_i)=r_i, \quad i=1,\ldots, n.
$$
Then we have the following two statements:
\begin{itemize}
\item The line $e_{i-1}e_{i}$ is not contained in the $2$-plane
  $r_{i-2}r_{i-1} r_{i+1}$;
\end{itemize}
Denote by $\hat e_{i-1}$ the (projective) intersection point of the line
  $e_{i-1}e_{i}$ and the $2$-plane $r_{i-2}r_{i-1} r_{i+1}$ (see
  Figure~\ref{fig:wu1}).  Then additionally we have:
\begin{itemize}
\item $\hat e_{i-1} \notin r_{i-2}r_{i-1}$;
\item $\hat e_{i-1} \notin r_{i-1} r_{i+1}$.
\end{itemize}
\end{proposition}

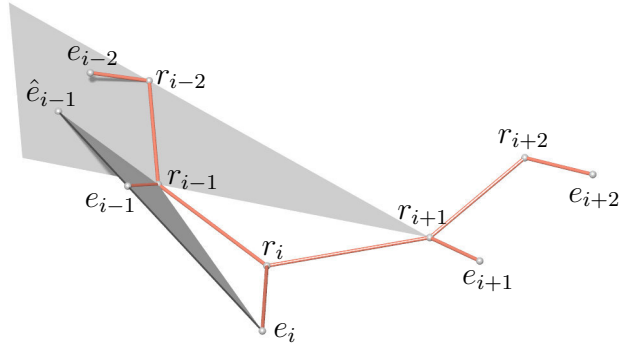
\begin{figure}[h]
  \centerline{
  \begin{overpic}[width=.62\textwidth]{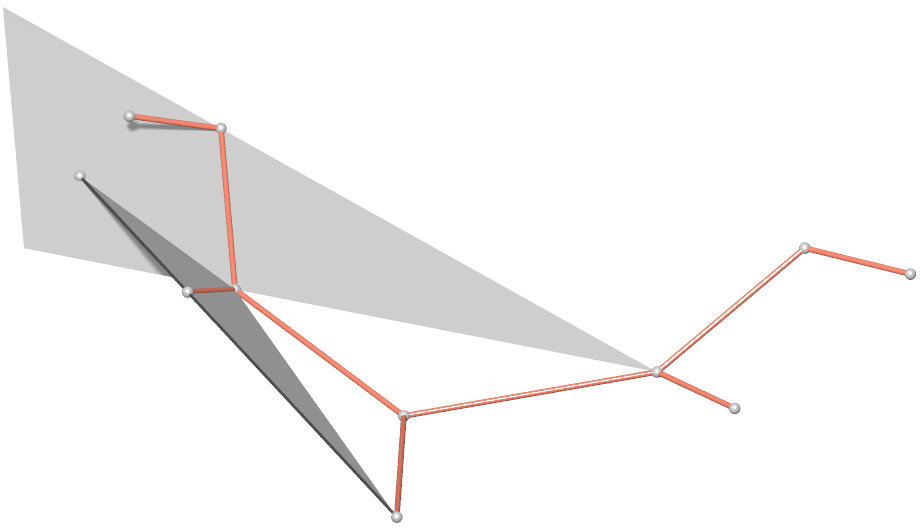}
    \put(10,47){$e_{i - 2}$}
    \put(13,22){$e_{i - 1}$}
    \put(45,0){$e_i$}
    \put(77,9){$e_{i + 1}$}
    \put(95,23){$e_{i + 2}$}
    \put(25,43){$r_{i - 2}$}
    \put(27,26){$r_{i - 1}$}
    \put(43,14){$r_i$}
    \put(66,20){$r_{i + 1}$}
    \put(83,33){$r_{i + 2}$}
    \put(3,40){$\hat e_{i - 1}$}
  \end{overpic}
}
  \caption{Definition of $\hat e_{i-1}$.
  }
  \label{fig:wu1}
\end{figure}

\begin{remark}
  The theory of tensegrities (or equivalently the theory of infinitesimal
  rigidity) is projectively invariant.  So we do not consider special
  cases of parallel objects.  They are not parallel after an appropriate
  choice of an affine chart.
\end{remark}

\begin{proof}[Proof of Proposition~\ref{Predefinition-HF}]
First of all, the point $e_{i-1}$ is not in the $2$-plane
  $$
  r_{i-2}r_{i-1} r_{i+1}
  $$
  (cf.\ Figure~\ref{fig:wu1}), as otherwise
$$
  \spann(r_{i-2},r_{i-1},r_{i})=\spann(r_{i-2},r_{i-1},e_{i-1})
=\spann(r_{i-2},r_{i-1},r_{i+1})
$$
(here the first equality holds as vectors $r_{i-2}e_{i-1}$, $ r_{i-1}e_{i-1}$ and $r_{i}e_{i-1}$ are adjacent to the same edge and, therefore, linear genericity of $C_B(P)$ implies that they are contained in a $2$-plane; the second equality holds by the assumption $e_{i-1}\in  r_{i-2}r_{i-1} r_{i+1}$)
which would imply that the points $r_{i-2},r_{i-1}, r_i,r_{i+1}$ are
  contained in a $2$-plane, and therefore $C(P)$
is not in a 3D-general flat position.
Therefore,
the line $e_{i-1}e_{i}$ is not in the $2$-plane $r_{i-2}r_{i-1} r_{i+1}$.

Secondly, if $\hat e_{i-1}\in r_{i-2}r_{i-1}$, then the points
$e_i,e_{i-1},r_{i-2},r_{i-1}$ are in a $2$-plane.
Now the point $r_i$ is in this $2$-plane as it is in the span of $e_{i-1},r_{i-1},r_{i-2}$;
and additionally $r_{i+1}$ is in this $2$-plane as it is in the span of $e_{i},r_{i-1},r_{i}$.
Therefore, $r_{i-2},r_{i-1},r_{i},r_{i+1}$ are in this $2$-plane, which contradicts to
flat 3D-genericity of the cycle.

Finally, if $\hat e_{i-1}\in r_{i-1}r_{i+1}$, then the points
$e_i,e_{i-1},r_{i-1},r_{i+1}$ are in a $2$-plane.
Now the point $r_i$ is in this $2$-plane as it is in the span of $e_i,r_{i-1},r_{i+1}$;
and additionally $r_{i-2}$ is in this $2$-plane as it is in the span of $e_{i-1},r_{i-1},r_{i}$.
Therefore, $r_{i-2},r_{i-1},r_{i},r_{i+1}$ are in this $2$-plane, which contradicts to
flat 3D-genericity of the cycle.

This concludes the proof of all statements of the proposition.
\end{proof}

For the definition of $\wu$-surgeries we need an index\dash symmetric
statement. The following corollary is just the index\dash symmetric version of
Proposition~\ref{Predefinition-HF}.

\begin{corollary}\label{Predefinition-HF2}
Let a framed cycle framework $C_B(P)=(C(P),B(P))$ be in a 3D-general flat position.
Let also
$$
G(P)(b_i)=e_i, \quad G(P)(c_i)=r_i, \quad i=1,\ldots, n.
$$
Then we have the following two statements:
\begin{itemize}
\item The line $e_{i}e_{i+1}$ is not in the $2$-plane $r_{i-1}r_{i+1}
  r_{i+2}$;
Denote by $\hat e_{i+1}$ the (projective) intersection point of the line
  $e_{i}e_{i+1}$ and the $2$-plane $r_{i-1}r_{i+1} r_{i+2}$.
Then additionally we have:
\item $\hat e_{i+1} \notin r_{i+1}r_{i+2}$;
\item $\hat e_{i+1} \notin r_{i-1}r_{i+1}$.
\end{itemize}
\end{corollary}

\begin{proof}
After swapping the indexes $i\to n-i$ for all $i$ in $C_B$
we arrive at the statement of Proposition~\ref{Predefinition-HF}
for $n-i$.
\end{proof}

\subsubsection{$\wu$-surgeries}

Let us continue with the definition of $\wu$\dash surgeries.

\begin{definition}
Consider a framed cycle
$$
C_B=\big((c_1,\ldots,c_n),(b_{1},\ldots,b_{n})\big),
$$
and its framework
$$
C_B(P)=\big((r_1,\ldots,r_n),(e_{1},\ldots,e_{n})\big)
$$
in 3D-general flat position.  Let $i \in 1, \ldots, n$.
The \emph{$\wu$-surgery} of the cycle $C$ at node $i$ is the cycle
\begin{align*}
\wu_i(C_B(P))
  =\big(
  &(r_1,\ldots,r_{i-2},r_{i-1},r_{i+1},r_{i+2},\ldots, r_n),\\
  &(e_1,\ldots,e_{i-2},\hat e_{i-1}, \hat e_{i+1},e_{i+2}, \ldots, e_n)\big),
\end{align*}
where
$$
\begin{array}{l}
\hat e_{i-1}= e_{i}e_{i-1} \cap r_{i+1}r_{i-1}r_{i-2};\\
\hat e_{i+1}= e_{i}e_{i+1} \cap r_{i-1}r_{i+1}r_{i+2},\\
\end{array}
$$
(see Figure~\ref{fig:wu2}).
\end{definition}

\begin{remark}
Due to Proposition~\ref{Predefinition-HF} and Corollary~\ref{Predefinition-HF2},
the points $\hat e_{i-1}$ and $\hat e_{i+1}$ are uniquely defined.
\end{remark}

\begin{figure}[h]
  \centerline{
  \begin{overpic}[width=.62\textwidth]{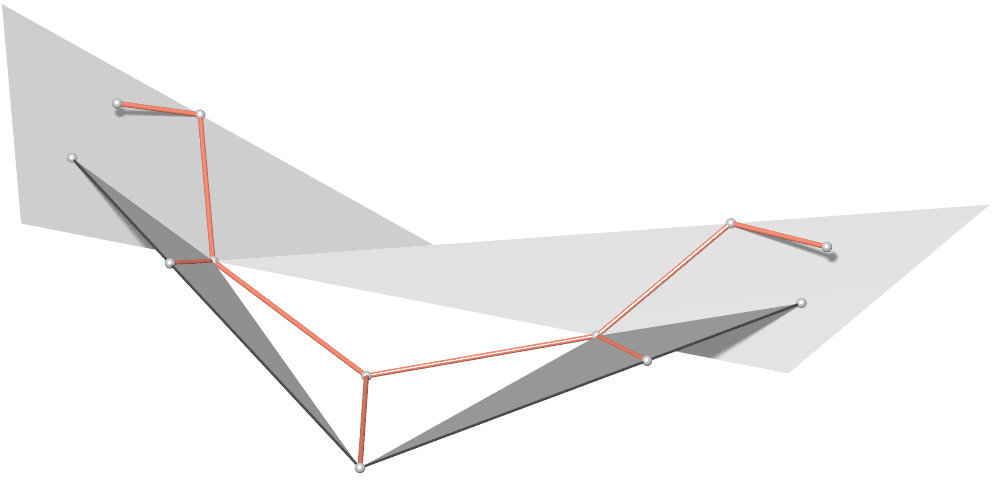}
    \put(6,40){$e_{i - 2}$}
    \put(9,19){$e_{i - 1}$}
    \put(33,-2){$e_i$}
    \put(65,9){$e_{i + 1}$}
    \put(86,22){$e_{i + 2}$}
    \put(21,38){$r_{i - 2}$}
    \put(23,23){$r_{i - 1}$}
    \put(36,12){$r_i$}
    \put(53,17){$r_{i + 1}$}
    \put(70,28){$r_{i + 2}$}
    \put(2,34){$\hat e_{i - 1}$}
    \put(82,15){$\hat e_{i + 1}$}
  \end{overpic}
}
  \caption{The construction of $\wu_i$-surgery. Here
  we exclude the vertices $e_i$ and $r_i$ and replace $e_{i-1}$ and
  $e_{i+1}$ respectively by $\hat e_{i-1}$ and $\hat e_{i+1}$.
  }
  \label{fig:wu2}
\end{figure}
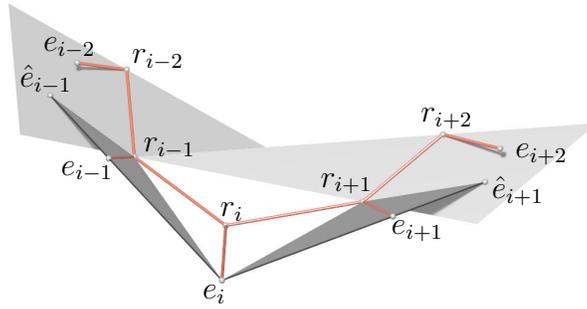

\begin{corollary}
A $\wu$-surgery takes a framework of a framed cycle in 3D-general flat position
to a framework of a framed cycle in 3D-general flat position.
\end{corollary}

\begin{proof}
The set of $C(P')$-vertices after the surgery is a subset of $C(P)$
therefore, there are no four points of $C(P')$ in a $2$-plane.

By construction we have
$$
  \hat e_{i-1}\in \spann(r_{i-2},r_{i-1},r_{i+1})
  \quad \hbox{and} \quad
  \hat e_{i+1}\in \spann(r_{i-1},r_{i+1},r_{i+2}).
$$

Further, by Proposition~\ref{Predefinition-HF} every two vectors from
$$
\{r_{i-1} - \hat e_{i-1}, r_{i-1} - r_{i-2}, r_{i-1} - r_{i+1}\}
$$
are not collinear.

Finally, by Corollary~\ref{Predefinition-HF2} every two vectors of
$$
\{r_{i+1} - \hat e_{i+1}, r_{i+1} - r_{i+2}, r_{i+1} - r_{i-1}\}
$$
are not collinear.
Therefore, $\wu_i(C_B(P))$ is in 3D-general flat position.
\end{proof}

\subsubsection{Static properties of $\wu$-surgeries}
We continue with the following important property of $\wu$-surgeries.

\begin{proposition}\label{wu-property}
Let $C_B$ be a framed cycle of length $m$ and $i\in\{1,\ldots, m\}$.
A framework $C_B(P)$ in 3D-general flat position admits a non-zero tensegrity if and only if
$\wu_i(C_B(P))$ admits a non-zero tensegrity for every admissible $i$.
\end{proposition}

\begin{proof}
Let $C_B(P)$ admit a tensegrity $G(C_B(P),w)$.
First, we construct a framed cycle tensegrity $(C_{B,i}^3(P),\hat w)$.
Let
$$
C_{B,i}^3(P)=
\big((r_{i-1},r_i,r_{i+1}),
(e_{i},e_i,e_{i})\big)
$$
(see Figure~\ref{fig:cb3}).

\begin{figure}[h]
  \centerline{
  \begin{overpic}[width=.42\textwidth]{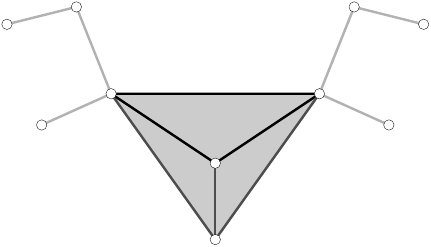}
    \put(20,54){$r_{i - 2}$}
    \put(25,38){$r_{i - 1}$}
    \put(47,24){$r_i$}
    \put(61,38){$r_{i + 1}$}
    \put(67,54){$r_{i + 2}$}
    \put(-1,46){$e_{i - 2}$}
    \put(4,23){$e_{i - 1}$}
    \put(42,0){$e_i$}
    \put(83,23){$e_{i + 1}$}
    \put(91,46){$e_{i + 2}$}
  \end{overpic}
}
  \caption{
    The framed cycle $C_{B, i}^3(P)$ is illustrated by the shaded area.
    Notice that the boundary points for this cycle all coincide with $e_i$.
  }
  \label{fig:cb3}
\end{figure}
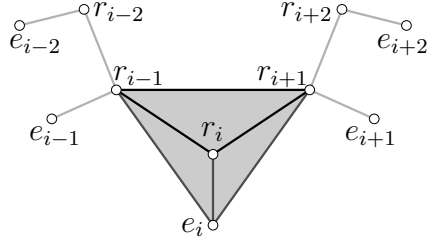

The stress $\hat w$ is defined from the following condition:
at all edges adjacent to $r_i$ the stress $\hat w$ coincides with the stress $w$ for $(C_B(P),w)$.
It is clear that the tensions $\hat w$ on the remaining edges of $C_{B,i}^3(P)$
are defined in the unique way.

Let us now subtract $(C_{B,i}^3(P),\hat w)$ from $(C_B(P),w)$.
We have:
\begin{itemize}
\item zero stresses at all vertices adjacent to $r_i$.
\item the sum of vectors of forces  $\lambda r_{i-1}e_{i-1}$ and $\mu r_{i-1}e_{i}$
(for some non-zero $\lambda$ and $\mu$) should be in the $2$-plane spanned by $r_{i-2},r_{i-1},r_{i+1}$
and therefore it is in the line $r_{i-1}\hat e_{i-1}$.
\item the sum of vectors of forces  $\lambda r_{i+1}e_{i+1}$ and $\mu r_{i+1}e_{i}$
(for some non-zero $\lambda$ and $\mu$) should be in the $2$-plane spanned by $r_{i-1},r_{i+1},r_{i+2}$
and therefore it is in the line $r_{i+1}\hat e_{i+1}$.
\end{itemize}

Since the points $r_{i-2},r_{i-1}, r_{i},r_{i+1}$ span a $3$-plane,
the $2$-planes $r_{i-1}e_{i-1}e_{i}$ and $r_{i-2}r_{i-1}r_{i+1}$ intersect by a line.

Symmetrically, the $2$-planes $r_{i+1}e_{i+1}e_{i}$ and $r_{i-1}r_{i+1}r_{i+2}$ intersect by a line.

Therefore, the resulting tensegrity is a non-zero tensegrity on the framework $\wu_i(C_B(P))$.

Now let us assume that there is a non-zero tensegrity on $\wu_i(C_B(P))$.
Then we consider a tensegrity $(C_{B,i}^3(P),\tilde w)$, where $C_{B,i}^3(P)$ is
the framed $3$-cycle framework as above;
the self-stress $\tilde w$ is defined by linearity starting from the fact that at edge
$r_{i-1}r_{i+1}$ it coincides with the self-stress at $r_{i-1}r_{i+1}$ for $\wu_i(C_B(P))$.

Similarly, by subtracting $(C_{B,i}^3(P),\tilde w)$ from $(\wu_i(C_B(P)),w)$ and summing
the boundary force vectors at $r_{i-1}$ and $r_{i+1}$ we get a non-zero
tensegrity for $C_B(P)$.
\end{proof}

\subsection{On existence and uniqueness of tensegrities for
frameworks in 3D-general position}
\label{On existence and uniqueness of tensegrities for
frameworks in general position}
Recall that in this paper we work only with connected graphs.
The uniqueness of tensegrities (up to a scalar) can be formulated as follows.

\begin{proposition}\label{non-zero}
All tensegrities on a linearly generic framework
are proportional.
In addition every non-zero tensegrity is everywhere non-zero.
\end{proposition}

\begin{proof}
The proof is straightforward as tensions at every vertex of a linearly generic framework
are defined in the unique way up to a scalar. All stresses at this vertex are either all zero or all non-zero.
\end{proof}

Before to formulate a criterion of existence of a tensegrity we give the following definition.

\begin{definition}\label{def:framed}
Consider a linearly generic framework $G(P)$ and a cycle $C$ in $G$ (without self-intersections).
Furthermore, consider a framed cycle $C_B=(C,B)$. We say that a framed cycle framework
$$
C_B(\tilde P)=
\big((r_{1},\ldots, r_{n}),
(e_{1},\ldots,e_{n})\big)
$$
is
  \emph{associated} to $G(P)$ if
\begin{itemize}
\item $r_i=p_i$ at all corresponding points of $C$ and $G$;
\item for the boundary points we have:
$$
\begin{array}{l}
e_i\in \spann(p_{i},p_{i-1},p_{i+1})\cap
\spann(p_{i},p_{i,1},\ldots, p_{i,k}),
\end{array}
$$
where $v_iv_{i,j}$ correspond to all edges adjacent to $v_i$ except
for the two edges $v_iv_{i-1}$ and $v_iv_{i+1}$, and $p_ip_{i,j}$ are their realizations in $G(P)$.
\item In addition we require that $e_i\ne r_i$ for $i=1,\ldots, n$.
\end{itemize}
\end{definition}

\begin{remark}
Let us note that associated framed cycles are very specific tie-downs introduced in~\cite{White1983}
by N.L.~White and W.~Whiteley. In the original construction of N.L.~White and W.~Whiteley there is no fixed rule
to pick the directions of tie-downs (i.e., framings), whereas in our construction the directions of framings are determined by
 the framework.

Indeed, consider a linearly generic framework $G(P)$ and a cycle $C$ in $G$.
Note that the vertices of an associated framed cycle $C_B$ are the vertices of $G$.
Now the directions of edges $e_ir_i$ are uniquely determined by the framework.
The only freedom $e_i$ can still have is as follows: it can slide along the line
$$
\begin{array}{l}
\spann(p_{i},p_{i-1},p_{i+1})\cap
\spann(p_{i},p_{i,1},\ldots, p_{i,k}).
\end{array}
$$
Here the position of $e_i$ on that line is not important as it does not change the force-loads on the cycle itself.
\end{remark}

The criterion of existence of a tensegrity can be formulated in the following way.

\begin{theorem}\label{cycle-tens}
A linearly generic connected framework admits a non-zero tense\-grity
if and only if all its associated framed cycle frameworks admit a non-zero
tense\-grity.
\end{theorem}

\begin{proof}
Assume that a framework admits a non-zero tensegrity.
Then the associated framed cycle frameworks admit a non-zero tensegrity directly by Proposition~\ref{non-zero}.

Let now all associated framed cycle frameworks of $G(P)$ admit a non-zero tensegrity.
Let us iteratively construct a non-zero tensegrity for $G(P)$.

We start with any vertex of degree greater than 1 and set the stress on
one of its edges to $1$.
Therefore, the stresses for the other edges are defined in the unique way.

Assume now that we have constructed the stresses for the edges adjacent to all vertices of $V'\subset V$.
In addition we assume that every pair of vertices in $V'$ is connected by a path in $G$ within $V'$.

Let us now consider some edge $v'v$ such that $v'\in V'$ and $v\in V\setminus V'$.
If $v\in B$ then there is no equilibrium condition on stresses, we just add $v$ to $V'$.
Let now $v'$ be $k$-valent ($k > 1$) with edges $vv_1,\ldots,vv_k$ adjacent to $v$.
Consider the following two cases for these edges.

{\noindent
  \emph{Case 1:} $v_i\notin V'$. Then the stress at $vv_i$ is not yet defined.
Hence we define it from the equilibrium condition for $v$.
}

{\noindent
  \emph{Case 2:} $v_i\in V'$. Then there exists
an associated framed cycle framework $C_B(\tilde P)$ whose non-boundary vertices all correspond to vertices in $V'\cup \{v\}$ and
that passes through $v$ and $v_i$ via edge $vv_i$.
First of all, it has a non-zero self-stress by the theorem assumption.
Secondly, this self-stress is proportional to the stresses defined on the edges adjacent to $V'$ (since all the vertices but one are in $V'$, and
the equilibrium conditions in $V'$ are fulfilled simultaneously for the self-stress on the cycle $C_B(\tilde P)$ and the partially constructed stress).
So the stress at $vv_i$ defined from $v_i$ before coincides with the stress at $vv_i$ defined by the equilibrium in $v$.
}

Now we add $v$ to $V'$ and continue to the next vertex of $V\setminus V'$.
Note that after adding $v$ to $V'$ all the vertices of the new $V'$
are connected by edge paths of $G$ via vertices of $V'$.

At each step of iteration we add a new vertex and define the stresses on the edges adjacent to it (if they were not defined before)
such that the equilibrium condition is fulfilled.

Since $G$ is connected, the process terminates and we have a tensegrity on $G(P)$ at the end of the process.
\end{proof}

\subsection{Join/intersection condition for the existence of tensegrities}
\label{subsec:projcond}

Finally, we have all tools to formulate geometric conditions for the
existence of non-zero tensegrities for frameworks in 3D-general position in
terms of join/intersection operations and conditions within projective geometry.
Let us first briefly recall the notions of join and intersection.

\subsubsection{Join/intersection operations and relations}

Let us briefly recall the notions of join and intersection within projective
geometry.
First of all the elements of projective geometry on $\R P^d$ are all
the $k$-planes of all possible dimensions $k\le d$.

There are two operations in projective geometry that are called
\emph{join} and \emph{intersection} operations and denoted by $\vee$ and $\wedge$,
respectively.

We will use the ``dimension-operator'' $\dim$ with respect to
\emph{projective} dimension of projective subspaces and refer to the
dimension of \emph{linear} subspaces by $\dim_l$.
Projective subspaces $\pi \subset \R P^d$ of dimension $k$ are represented
by $(k + 1)$\dash dimensional subspaces $U \in \R^{d + 1}$, consequently,
$\dim(\pi) = \dim_l(U) - 1$.

\begin{definition}
  Given projective subspaces $\pi_1,\ldots, \pi_n \subset \R P^d$ of
  arbitrary dimensions.
  The \emph{join} and \emph{intersection} operations for these subspaces
  are respectively as follows:
  $$
  \begin{array}{l}
    \pi_1\vee \ldots \vee\pi_n := \spann(U_1 \cup \ldots \cup U_n);\\
    \pi_1\wedge \ldots \wedge\pi_n := \bigcap\limits_{i=1}^{n}\pi_i.
  \end{array}
  $$
\end{definition}

\begin{remark}
  Note that there is another approach to our problem using
  bracket algebra and Grassmann\dash Cayley algebra.
  It is developed, e.g.,
  in~\cite{Doubilet1974,Li2008,sturmfels-1993,white-1995,White1983}.
  The expressions in bracket algebra are written as polynomials of brackets
  (i.e., minors of a matrix) while the translation of our approach to
  the bracket ring would result in systems of simple brackets equal to
  zero.  We refer an interested reader to the above mentioned papers.
\end{remark}

Finally, let us formulate relations on the subspaces of projective
geometry.

\begin{definition}
  Given projective subspaces $\pi_1,\ldots, \pi_n$. We say that
  $$
  \pi_1\wedge \ldots \wedge\pi_n=\hbox{true},
  $$
  if there exist projective subspaces $\pi_i'$ with $\dim \pi_i'=\dim \pi_i$
  for $i=1,\ldots, n$
  such that
  $$
  \dim(\pi_1\wedge \ldots \wedge\pi_n)>\dim(\pi_1'\wedge \ldots \wedge\pi_n').
  $$
  Otherwise we say that
  $$
  \pi_1\wedge \ldots \wedge\pi_n=\hbox{false}.
  $$
Here we consider the dimension of an empty set to be $-1$.
  \end{definition}

\begin{example}
  Consider three projective  lines $\ell_1,\ell_2,\ell_3$ in a two-dimensional projective space.  Then
  $$
  \ell_1\wedge \ell_2 \wedge \ell_3=\hbox{true},
  $$
  if and only if these three lines have projectively at least one point in
  common.
\end{example}

\subsubsection{Join/intersection conditions for framed cycles}

Let us first start with a join/intersection condition for a framed cycle on
three vertices.

\begin{definition}
Let $C_B(P)$ be a framework of a framed cycle in 3D-general position
$$
\big((r_1,r_2,r_3), (e_1,e_2,e_3) \big).
$$
  Then the \emph{join/intersection condition} for $C$ is
$$
r_1e_1\wedge r_2e_2\wedge r_3e_3= \hbox{true}.
$$
\end{definition}

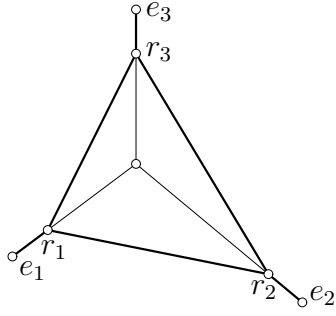
\begin{figure}[h]
  \centerline{
  \begin{overpic}[width=.32\textwidth]{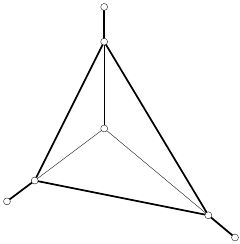}
    \put(12,18){$r_1$}
    \put(80,5){$r_2$}
    \put(46,82){$r_3$}
    \put(5,11){$e_1$}
    \put(99,3){$e_2$}
    \put(46,95){$e_3$}
  \end{overpic}
}
  \caption{
    A framed cycle consisting of a triangle $r_1, r_2, r_3$ with external
    forces $w_i (e_i - r_i)$ is a tensegrity if and only if the three
    lines $r_i e_i$ meet in a point.
  }
  \label{fig:triangle2}
\end{figure}

Let us now expand the notion of join/intersection condition to
framed cycle frameworks of arbitrary length.

\begin{definition}\label{GC-condition-cycle-def}
Let $C_B$ be a framed cycle of length $n\ge 3$, and let $C_B(P)$ be its framework in 3D-general position.
  Then the \emph{join/intersection condition} for $C_B(P)$ is as follows
$$
r_1e_1^{(n-3)}\wedge r_2e_2\wedge r_3e_3^{(n-3)}=\hbox{true},
$$
where $e_1^{(n-3)}$ is defined recursively by
$$
\begin{array}{l}
e_1^{(0)}=e_1;\\
e_1^{(k)}=e_{n-k+1}e_1^{(k-1)} \wedge \big(r_{n-k}\vee r_1\vee r_{2}\big),
\end{array}
$$
and $e_3^{(n-3)}$ is defined recursively by
$$
\begin{array}{l}
e_3^{(0)}=e_n;\\
e_3^{(k)}=e_{n-k}e_{3}^{(k-1)} \wedge \big(r_{n-k-1} \vee r_{n-k}\vee r_1\big).
\end{array}
$$
\end{definition}

\begin{remark}
For simplicity here and below we write $uw$ instead of $u\vee v$.
\end{remark}

\begin{proposition}\label{GC-condition-cycle}
A framed cycle framework $C_B(P)$ in 3D-general flat position has a non-zero tensegrity
if and only if
$C_B(P)$ fulfills the join/intersection condition.
\end{proposition}

\begin{proof}
  The condition is written by iteratively application of $\wu$\dash
  surgeries to the last vertex of $C$, reducing $C_B$ to a triangular
  framed cycle in general flat position.  Namely the resulting flat cycle
  is
  $$
    \wu_{4}(\ldots \wu_n(C_B(P))\ldots).
  $$

  The existence of a non-zero tensegrity is equivalent to the existence of
  a non-zero tensegrity after $\wu$-surgeries by
  Proposition~\ref{wu-property}.

  So the statement of proposition is reduced to triangular cycles.
  The statement for a triangular cycle (which has to be planar) is
  classical (see e.g.,~\cite{Karpenkov2019}).
\end{proof}

Let us write explicitly the join/intersection conditions for cycles on~3
and~4 vertices.

\begin{example}
  \label{ex:gc}
If $n=3$, then we have
$$
r_1e_1\wedge r_2e_2\wedge r_3e_3=\hbox{true}.
$$
If $n=4$, then we have (see Figure~\ref{fig:n4})
$$
\big[r_1\vee\big(e_4e_1 \wedge (r_{3}\vee r_1 \vee r_2)\big)\big]
\wedge r_2e_2
\wedge
\big[r_3\vee\big(e_3e_4 \wedge (r_{2}\vee r_3 \vee r_1)\big)\big]
=\hbox{true}.
$$
\end{example}

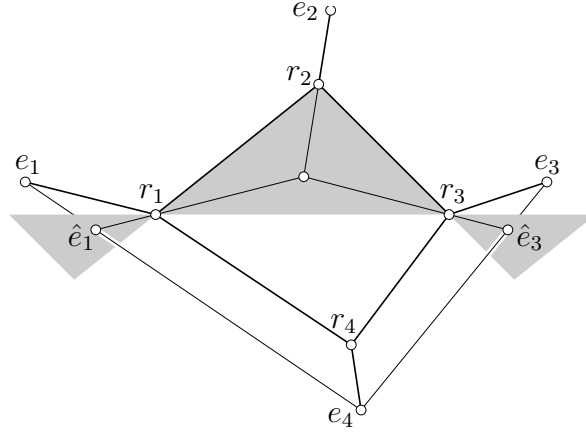
\begin{figure}[h]
  \centerline{
  \begin{overpic}[width=.62\textwidth]{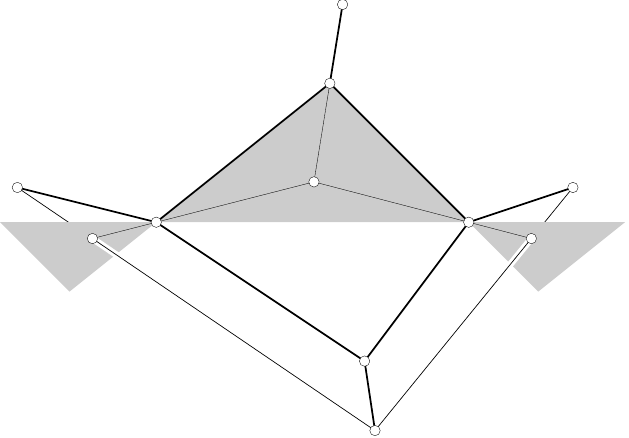}
    \put(22,37){$r_1$}
    \put(47,57){$r_2$}
    \put(73,37){$r_3$}
    \put(54.5,15){$r_4$}
    \put(1,42){$e_1$}
    \put(48,68){$e_2$}
    \put(89,42){$e_3$}
    \put(54,-1){$e_4$}
    \put(10,29){$\hat e_1$}
    \put(86,29){$\hat e_3$}
  \end{overpic}
}
  \caption{Illustration of Example~\ref{ex:gc} for $n = 4$.
  }
  \label{fig:n4}
\end{figure}

\subsubsection{Join/intersection criteria for tensegrities in 3D-general position}

The following theorem and its proof is the recipe to write the
join/intersection criteria for tensegrities in 3D-general position.

\begin{theorem}\label{GC-conditions}
  The framework $G(P)$ in 3D-general position admits a non-trivial
  tensegrity, if and only if all the join/intersection conditions for all
  its associated framed cycle frameworks are fulfilled.
\end{theorem}

\begin{proof}
  The join/intersection conditions for $G(P)$ are written according to
  Definition~\ref{GC-condition-cycle-def}.
  Due to Theorem~\ref{cycle-tens} and Proposition~\ref{GC-condition-cycle}
  they are equivalent to the existence of a non-zero tensegrity on $G(P)$.

  It remains to add the following detail to the above construction.
  In order to generate the boundary $\tilde e_i$ of an associated framed
  cycle framework $C_B(\tilde P)$, one should take the intersection of the
  span of two edges in the cycle passing through $\tilde r_i=r_i$
  (namely $r_{i-1}r_i$ and $r_ir_{i+1}$) and the span of all other edges
  adjacent to $r_i$, say $r_ir_{i,1},\ldots, r_ir_{i,k}$.
  Let us denote the resulting line by $\ell$.
  In terms of join/intersection operators $\ell$ is written as
  $$
    \ell=\big(r_i\vee r_{i,1}\vee \ldots \vee r_{i,k}\big)\wedge\big(r_{i-1}\vee
    r_i \vee r_{i+1}\big).
  $$
  Finally, we pick up a point $G(P)(b_i)$ on $\ell$ distinct to $r_i$.
  For instance, set
  \begin{equation*}
    \tilde e_i=\ell \wedge \big(r_{i,1}\vee \ldots \vee r_{i,k}\big).
    \qedhere
  \end{equation*}
\end{proof}

\begin{remark}
 In analogy to Remark~\ref{rem:h1} it is sufficient to check the
  criterion only for generator loops of the first homology group $H_1(G)$
  of the graph (if we consider the graph as topological space), because
  all other loops can be decomposed by those.
  These conditions for different generators of  $H_1(G)$ may still coincide.
  Some of the conditions will correspond to different strata.
\end{remark}

\section{Tensegrities and discrete harmonic maps}
\label{sec:applications}

In this section we relate ten\-se\-grities to the notion of discrete harmonic
functions and demonstrate an alternative way to obtain tensegrities with
just positive tensions.

The discrete Laplace operator (the graph Laplacian) acts on
maps $f: G \to \R^d$ defined on arbitrary graphs $G$ by with real valued
weights $w_{i, j} \in \R$
$$
(\Delta f)(v_i) := \sum_{v_j \sim v_i} w_{i, j} (f(v_i) - f(v_j)),
$$
where we sum over neighboring vertices $v_j$ of $v_i$.
This \emph{discrete Laplace operator} has been used in several
applications of geometry processing as well as in discrete complex
analysis and discrete minimal surface theory
(see e.g., \cite{bobenko+2008,botsch+2010}).
The weights $w_{i, j} \in \R$ are chosen depending on the application.
Prominent examples are the cotangent-weights or the area of Voronoi cells
around the vertex $v_i$. Furthermore, the choice of the weights implies
which properties of the discrete Laplace operator ``inherits'' from its
smooth counterpart~\cite{wardetzky+2007}.

\begin{definition}
A function $f: G \to \R^d$ is called \emph{discrete harmonic} if
$$
(\Delta f)(v_i) = 0
$$
for all vertices $v_i \in Z(G)$.
\end{definition}

A real valued discrete
harmonic function over some rectangular subgrid of the $\Z^2$ lattice is
illustrated by Figure~\ref{fig:example1}.

In the setting of tensegrities the function $f$ describes the coordinates
of the position of the vertices in space and the weight assignment $w_{i,
j}$ represents the stress at each edge $(v_i; v_j)$. Consequently, we will
allow positive and negative weights for tensile and compression forces.
Therefore, it follows from the definition of tensegrities
(Definition~\ref{defR}) that they can be seen as zeroes of the discrete
Laplace operator for maps defined on the vertices of a graph.  In this
sense tensegrities are harmonic maps with respect to the discrete Laplace
operator.

\begin{figure}[t]
  \centerline{
  \begin{overpic}[width=.28\textwidth]{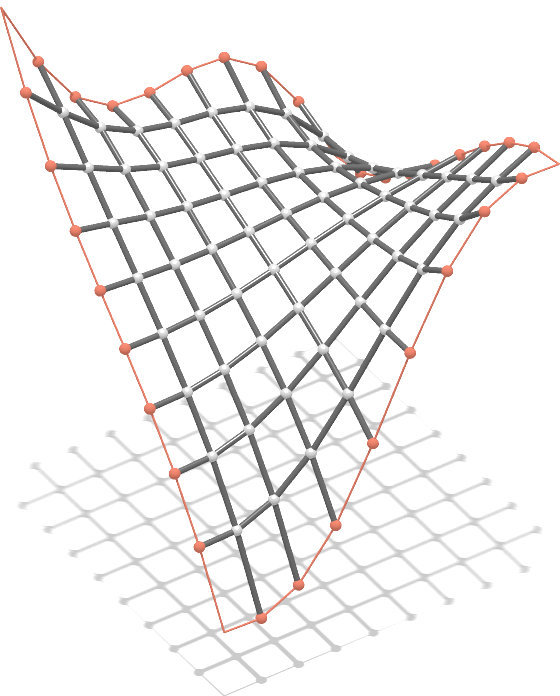}
  \end{overpic}%
}
  \caption{Illustration of a discrete harmonic real valued function as
  solution of a Dirichlet boundary value problem. The discrete harmonic
  function $f$ is defined over a rectangular patch $U$ of the $\Z^2$
  lattice $f: \Z^2 \supset U \to \R$ with weights $w = 1$ at each edge.
  The values of $f$ can be computed by minimizing an energy (cf.\
  Prop.~\ref{prop:energy}).
  }
  \label{fig:example1}
\end{figure}

\begin{proposition}
  \label{prop:energy}
  Tensegrities with arbitrary combinatorics and with only positive
  tensions $w_{i, j} > 0$ can be obtained as the minimum of the discrete
  Dirichlet energy
  \begin{equation}
    \label{eq:energy}
    \sum_{v_i \in G} \sum_{v_j \sim v_i} w_{i, j} \|p_i - p_j\|^2,
  \end{equation}
  viewing the coordinates of the vertices $p_i$ as variables.
\end{proposition}
This proposition holds since the energy is bounded below and
critical points are solutions of the linear system
\begin{equation}
  \label{eq:linearsystem}
  \sum_{v_j \sim v_i} w_{i, j} (p_i - p_j) = 0,
\end{equation}
for all $i$. So if the number of vertices is big enough and the boundary
vertices are fixed, the minimum is unique. Therefore, a tensegrity with
just positive tensions can be interpreted as critical point of an energy.

\begin{example}
  Let us consider a rectangular patch $U$ and let us further fix the
  values on the boundary of $U$. We are looking for a harmonic function
  $f : U \subset \Z^2 \to \R$ with respect to the discrete Laplacian
  with constant positive weights that solves this Dirichlet problem.
  According to Proposition~\ref{prop:energy} we find the solution by
  minimizing
  $$
  \sum_{v_i \in U} \sum_{v_j \sim v_i} w_{i, j} \|f_i - f_j\|^2,
  $$
  where $f_i$ are considered as variables of this energy function.
  The values $f_i$ which belong to the minimum are the values of the
  harmonic function $f$ at $v_i$.
  We illustrate the graph $(v_i, f_i) \in \R^2 \times \R$ of a harmonic
  function $f$ in Figure~\ref{fig:example1}.
\end{example}

\begin{example}
  \label{ex:disc}
  For the combinatorics of any cell decomposition of a disc we obtain a
  tensegrity with everywhere unit tensions by fixing the
  positions of the boundary vertices and minimizing the quadratic energy
  in Equation~\eqref{eq:energy}. The tensegrity is the solution to the
  linear system~\eqref{eq:linearsystem}. An illustration of such a cell
  decomposition can be found in Figure~\ref{fig:examples} (left).
  To check whether this framework is a tensegrity with the machinery
  provided by Section~\ref{sec:ratios} or Section~\ref{sec:gca} requires
  to check the ``ratio'' condition or the ``join/intersection'' condition
  for a set of cycles that generates the first homology group $H_1(G)$.
\end{example}

\begin{example}
  \label{ex:strip}
  Figure~\ref{fig:examples} (right) illustrates a twisted strip
  represented by a net with regular quadrilateral combinatorics which is
  attached to two interlinked circles.
  The topology or the combinatorics of the graph does not play any role in
  the analysis of a framework whether it is a tensegrity. Neither, the
  ``global'' topology nor combinatorics of the framework is of importance
  in the Equations~\eqref{eq:energy} and~\eqref{eq:linearsystem}, just the
  local combinatorics of the vertex stars.
\end{example}

\begin{figure}[h]
  \hfill
  \begin{overpic}[width=.42\textwidth]{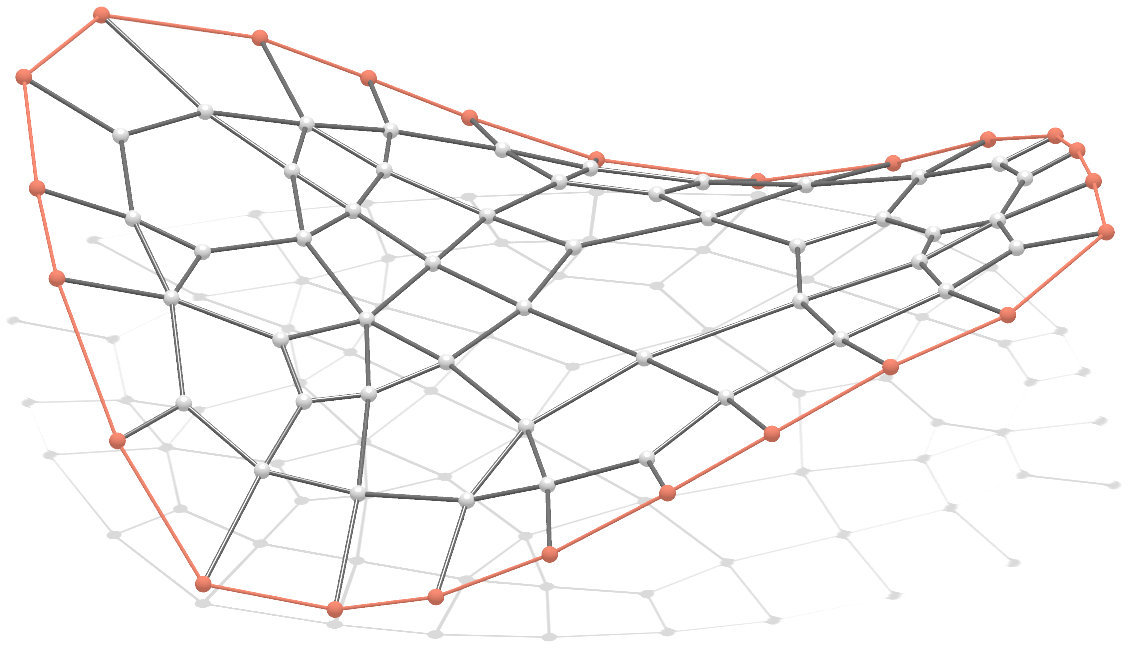}
  \end{overpic}%
  \hfill
  \begin{overpic}[width=.42\textwidth]{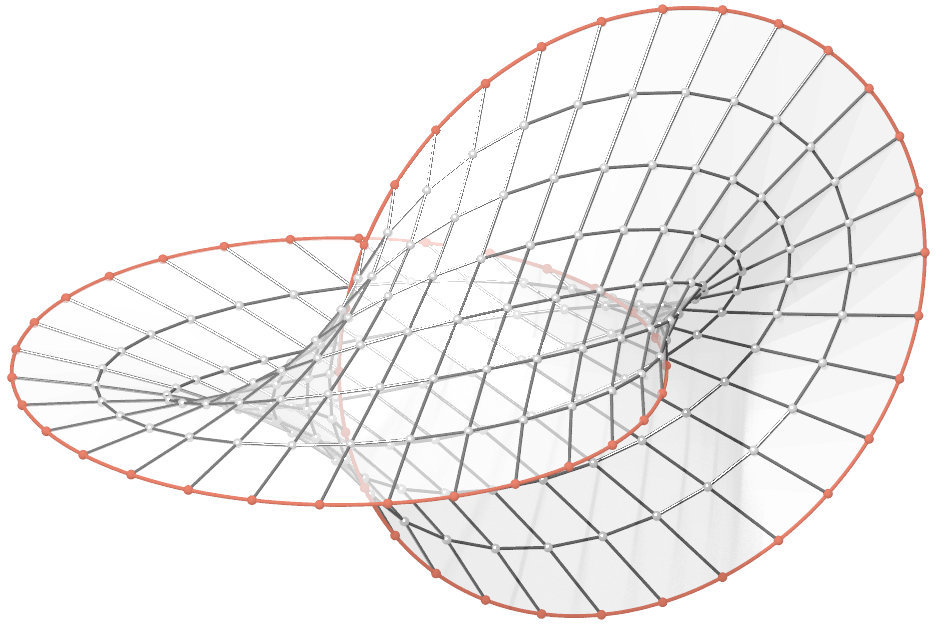}
  \end{overpic}
  \hfill{}
  \caption{
  \emph{Left}: A tensegrity with the combinatorics of an arbitrary cell
  decomposition of a disc (cf.\ Example~\ref{ex:disc}).
    \emph{Right}: Two circles are the boundaries of a tensegrity with
  regular quadrilateral combinatorics (cf.\ Example~\ref{ex:strip}).
  }
  \label{fig:examples}
\end{figure}

\section{Examples}
\label{sec:octahedron}

\begin{figure}[h]
  \hfill
  \begin{overpic}[width=.30\textwidth]{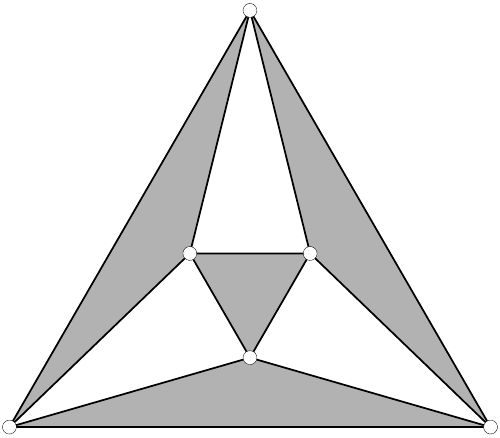}
    \put(64,36){\small$a$}
    \put(32,36){\small$b$}
    \put(47,6){\small$c$}
    \put(0,-6){\small$a'$}
    \put(101,0){\small$b'$}
    \put(52,83){\small$c'$}
  \end{overpic}%
  \hfill
  \begin{overpic}[width=.55\textwidth]{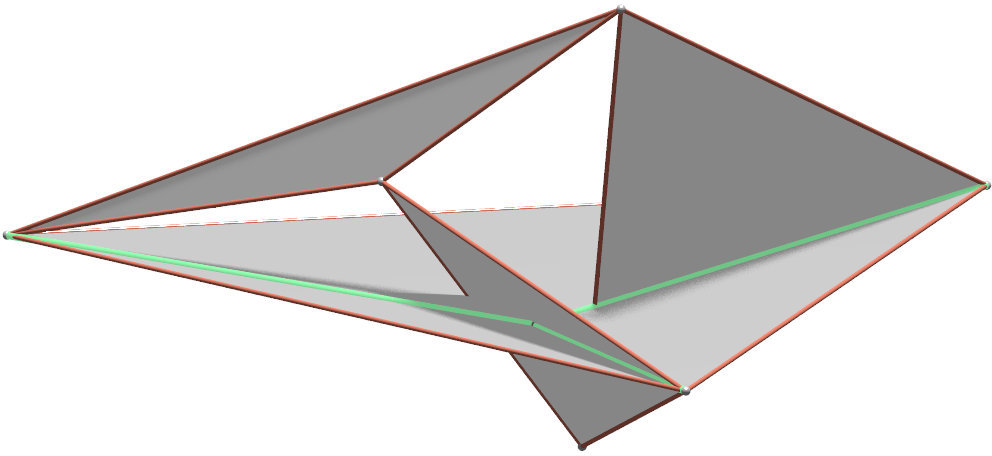}
    \put(-1,18){\small$a$}
    \put(71,3){\small$b$}
    \put(98,23){\small$c$}
    \put(53,-2){\small$a'$}
    \put(64,45){\small$b'$}
    \put(42,26){\small$c'$}
  \end{overpic}%
  \hfill{}
  \caption{\emph{Left}: The combinatorics of an octahedron.
  \emph{Right}: An octahedron in $\R^3$. Its edges form a tensegrity if
  and only if any four alternate face planes, i.e., four $2$-planes of the
  configurational type $abc$, $ab'c'$, $a'bc'$, $a'b'c$, are concurrent in
  a point. The intersection lines of the first $2$-plane with the three
  latter $2$-planes are illustrated by the green lines. See also
  Proposition~\ref{last}.
  }
  \label{fig:octahedron}
\end{figure}

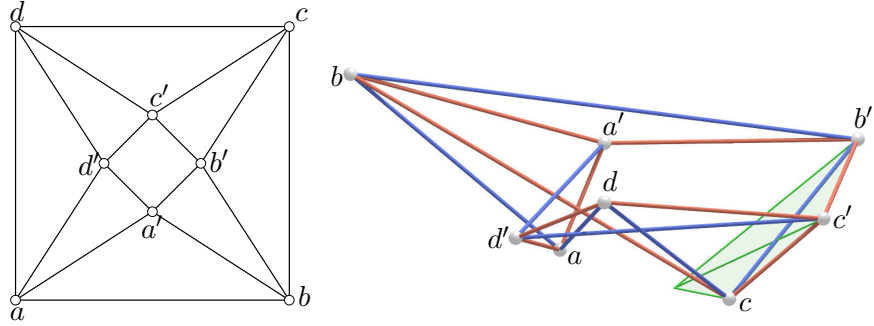
\begin{figure}[h]
  \hfill
  \begin{overpic}[width=.30\textwidth]{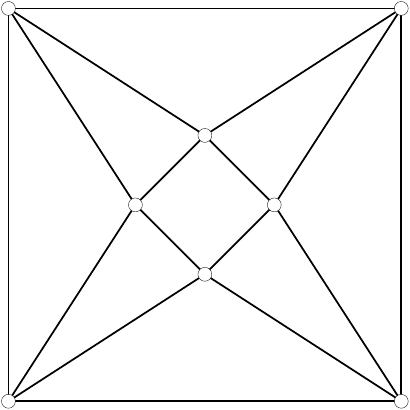}
    \put(0,-5){\small$a$}
    \put(101,0){\small$b$}
    \put(100,100){\small$c$}
    \put(0,100){\small$d$}
    \put(46,24){\small$a'$}
    \put(70,47){\small$b'$}
    \put(49,70){\small$c'$}
    \put(24,46){\small$d'$}
  \end{overpic}%
  \hfill
  \begin{overpic}[width=.55\textwidth]{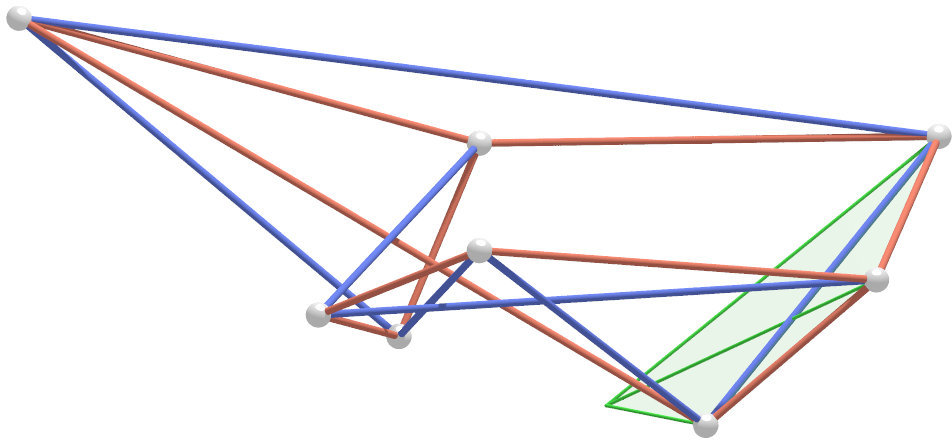}
    \put(43.2,8){\small$a$}
    \put(-2,42){\small$b$}
    \put(76,-1){\small$c$}
    \put(50,22){\small$d$}
    \put(50,32.7){\small$a'$}
    \put(98,34){\small$b'$}
    \put(94,15){\small$c'$}
    \put(28,11){\small$d'$}
  \end{overpic}%
  \hfill{}
  \caption{\emph{Left}: The combinatorics of a four\dash sided
  antiprism.
  \emph{Right}: A tensegrity in $\R^3$ with the combinatorics of a
  four\dash sided antiprism. The rods (= edges with positive weights) are
  red, the cables (= edges with negative weights) are blue.
  For the framework to be a tensegrity we have to check cycles with three
  and four vertices (see Example~\ref{ex:antiprism}).  For example the
  condition on the cycle $cc'b'$ is that the $2$-planes
  $cc'b'$, $bcd$, $a'bb'$, $c'dd'$, must be concurrent. The intersection
  lines of the first $2$-plane with the three latter $2$-planes are
  illustrated by the green lines.
  }
  \label{fig:antiprism}
\end{figure}

We conclude the paper with a brief description of two non-trivial
three\dash dimensional examples.
In fact however, our methods can be applied to all linearly generic
graphs (Def.~\ref{defn:lingen}), or in 3D-general position
(Def.~\ref{def:genpos}), respectively.
For graphs where the first homology group $H_1(G)$ can be generated from
cycles of length three or four, the conditions to check are written out in
Example~\ref{ex:gc} explicitly.

\begin{proposition}\label{last}
  An octahedral framework $(a,b,c,a',b',c')$ in $\R^3$ is a tensegrity if
  and only if four \emph{alternate}, i.e., face planes in a combinatorial
  configuration like $abc$, $ab'c'$, $a'bc'$, $a'b'c$, are concurrent in a
  point (see Figure~\ref{fig:octahedron} and cf.~\cite{White1987}).
\end{proposition}

Let us give two new proofs of this classical statement in terms of
multiplicative $1$-forms and in terms of join/intersection operations.

\begin{proof}[Proof 1 (via multiplicative $1$-forms)]
  Let us consider the cycle with three vertices $(a, b, c)$.
  The necessary condition for that cycle to be part of a tensegrity is
  that the product of the three values of the $1$-form multiply to $1$
  which is equivalent to Ceva's theorem (see
  Section~\ref{subsec:specialcase} for $n = 3$).
  Consequently, the three lines
  \begin{align*}
    \spann(a, b', c') &\cap \spann(a, b, c),
    \\
    \spann(a', b, c') &\cap \spann(a, b, c),
    \\
    \spann(a', b', c) &\cap \spann(a, b, c),
  \end{align*}
  must intersect in one point and therefore all four $2$-planes intersect in
  one point.
\end{proof}

\begin{proof}[Proof 2 (within join/intersection relations)]
  Let
$$
\begin{array}{l}
\ell_1=ab'c'\wedge abc;\\
\ell_2=a'bc'\wedge abc;\\
\ell_3=a'b'c\wedge abc.
\end{array}
$$
  Our condition for a triangle $abc$ is (cf.~\ref{ex:gc} for $n = 3$)
$$
\ell_1\wedge\ell_2\wedge \ell_3=\hbox{true}.
$$
This is to say that $b'c'a$, $c'a'b$, $a'b'c$, and $abc$ indeed meet in a
point.
\end{proof}

\begin{remark}
As one can notice, one can apply proofs~1 and~2 of Proposition~\ref{last} to any other triangle in the octahedron.
In fact all these conditions would be equivalent.
\end{remark}

\begin{example}
  \label{ex:antiprism}
  Let us consider a graph with the combinatorics of a four\dash sided
  antiprism (Figure~\ref{fig:antiprism} left).
  To determine whether a framework with such combinatorics is a tensegrity
  involves checking cycles with three and four vertices. Cycles
  with three vertices have been considered also in the previous example
  (Proposition~\ref{last}). The configuration for the cycles with four
  vertices is written down explicitly in Example~\ref{ex:gc}.
  An illustration of a tensegrity with the combinatorics of a four\dash
  sided antiprism is depicted in Figure~\ref{fig:antiprism} (right).
  The rods (= edges with positive weights) are depicted in
  red, the cables (= edges with negative weights) are depicted in blue.
\end{example}

\begin{example}
Let us decompose a two-dimensional domain $D$ homeomorphic to a disk into $k$ cells
that are either triangles or quadrilateral.
Let $G$ be the graph corresponding to the 1-skeleton of this decomposition.
Then
$$
 H_1(G)=\Z_k.
$$
Finally we consider frameworks in $\R^3$ representing $G$ (in 3D-general flat position).
One can pick all triangular and quadrilateral cycles of $G$ corresponding to all triangles
and quadrilaterals in the decomposition of $D$.
All the conditions for these cycles will be of two types described in
Example~\ref{ex:gc}. (Here one should substitute suitable vertices of the cycles to the corresponding expressions for $r_i$ and $e_i$.
Recall that the vertices $r_i$ are the corresponding vertices of the
  graph, and the vertices $e_i$ are defined from Definition~\ref{def:framed}.)
 \end{example}

\begin{example}
 Let us consider the example of the quadrilateral graph $G(m,n)$ on the
  torus with sides $m$ and $n$ (cf.\ Figure~\ref{fig:torus}).
 We have:
 $$
  H_1(G(m,n))=\Z_{mn+1}.
 $$
 One can pick $mn-1$ quadrilateral cycles, one longitude cycle, and one latitude cycle.
 All conditions of the quadrilateral cycles will be of the second type  described in
Example~\ref{ex:gc}.
The conditions for longitude and latitude cycles will be similar to the ones described in Example~\ref{ex:gc} but longer
(they are constructed by the iterations of Definition~\ref{GC-condition-cycle-def}).
\end{example}

\begin{figure}
  \begin{overpic}[width=.24\textwidth]{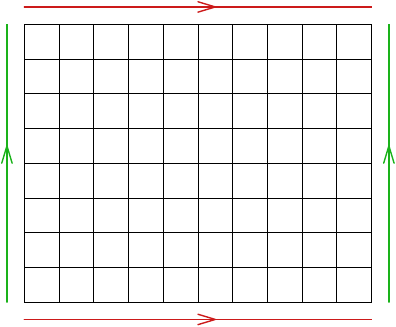}
  \end{overpic}
  \hfill
  \begin{overpic}[width=.24\textwidth]{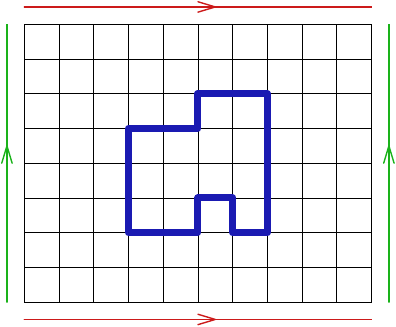}
  \end{overpic}
  \hfill
  \begin{overpic}[width=.24\textwidth]{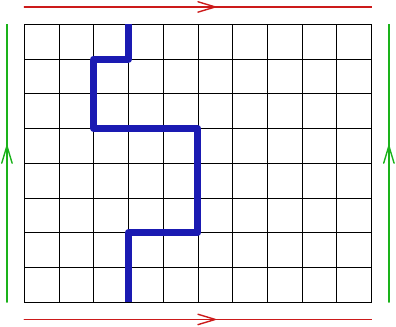}
  \end{overpic}
  \hfill
  \begin{overpic}[width=.24\textwidth]{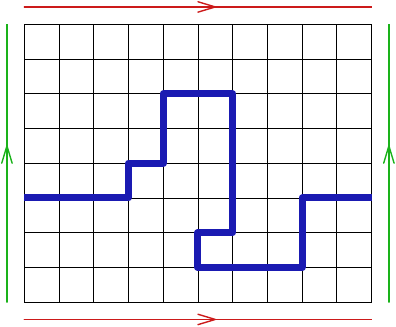}
  \end{overpic}
  \caption{Illustration of a torus.
  \emph{Left}: A torus is homeomorphic to a rectangle with opposite sides
  identified as illustrated. We are considering a framework with the
  combinatorics of the vertices of a rectangular sub-patch of the $\Z^2$
  lattice, with edges connecting neighbouring vertices, and with opposite
  vertices of the rectangle glued together like a torus.
  \emph{Second to third}: Illustrations of the three different types of
  cycles on that torus.
  }
  \label{fig:torus}
\end{figure}

\begin{remark}
In the examples of Figure~\ref{fig:examples}:
for the left one it is enough to write triangular, quadrilateral,
  pentagonal, and hexagonal conditions;
the right picture is very similar to the torus. Here one can pick mostly
quadrilateral cycles, one longitude cycle  and one latitude cycle.
More generally, this technique is applicable for all  graphs with linearly generic frameworks discussed in this paper (see Definition~\ref{defn:lingen}).
\end{remark}

\section*{Acknowledgments}

  Oleg Karpenkov is partially supported by the Engineering and Physical
  Sciences Research Council (EPSRC) grant EP/N014499/1 (LCMH).
  Christian M{\"u}ller gratefully acknowledges the support of the
  Austrian Science Fund (FWF) through project P~29981 and the LMS 2 grant
  scheme (Ref. No: 21810).
  Furthermore, we would like to thank the anonymous reviewers for their
  valuable comments.

\bibliographystyle{plain}
\bibliography{tensegrities-final-nonjournalstyle.bib}

\def\Yu{Yu}
\begin{thebibliography}{10}

\bibitem{Aloui2018}
Omar Aloui, David Orden, and Landolf Rhode-Barbarigos.
\newblock Generation of planar tensegrity structures through cellular
  multiplication.
\newblock {\em Appl. Math. Model.}, 64:71--92, 2018.

\bibitem{bobenko+2008}
Alexander~I. Bobenko and {\Yu}ri~B. Suris.
\newblock {\em Discrete differential geometry. Integrable structure}, volume~98
  of {\em Graduate Studies in Mathematics}.
\newblock American Mathematical Society, 2008.

\bibitem{botsch+2010}
Mario Botsch, Leif Kobbelt, Mark Pauly, Pierre Alliez, and Bruno Levy.
\newblock {\em Polygon Mesh Processing}.
\newblock AK Peters, 2010.

\bibitem{Caspar2013}
Donald~L. Caspar and Aaron Klug.
\newblock Physical principles in the construction of regular viruses.
\newblock {\em Cold Spring Harbor symposia on quantitative biology}, 27:1--24,
  02 1962.

\bibitem{Connelly2013}
Robert Connelly.
\newblock Tensegrities and global rigidity.
\newblock In {\em Shaping space}, pages 267--278. Springer, New York, 2013.

\bibitem{Connelly1996}
Robert Connelly and Walter Whiteley.
\newblock Second-order rigidity and prestress stability for tensegrity
  frameworks.
\newblock {\em SIAM J. Discrete Math.}, 9(3):453--491, 1996.

\bibitem{Guzman2006}
Miguel de~Guzm\'{a}n and David Orden.
\newblock From graphs to tensegrity structures: geometric and symbolic
  approaches.
\newblock {\em Publ. Mat.}, 50(2):279--299, 2006.

\bibitem{Doray2010}
Franck Doray, Oleg Karpenkov, and Jan Schepers.
\newblock Geometry of configuration spaces of tensegrities.
\newblock {\em Discrete Comput. Geom.}, 43(2):436--466, 2010.

\bibitem{Doubilet1974}
Peter Doubilet, Gian-Carlo Rota, and Joel Stein.
\newblock On the foundations of combinatorial theory. {IX}. {C}ombinatorial
  methods in invariant theory.
\newblock {\em Studies in Appl. Math.}, 53:185--216, 1974.

\bibitem{Ingber2014}
Donald~E. Ingber, Ning Wang, and Dimitrije Stamenovi\'{c}.
\newblock Tensegrity, cellular biophysics, and the mechanics of living systems.
\newblock {\em Rep. Progr. Phys.}, 77(4):046603, 21, 2014.

\bibitem{Jackson2015a}
Bill Jackson and Anthony Nixon.
\newblock Stress matrices and global rigidity of frameworks on surfaces.
\newblock {\em Discrete Comput. Geom.}, 54(3):586--609, 2015.

\bibitem{Karpenkov2018}
Oleg Karpenkov.
\newblock Open problems on configuration spaces of tensegrities.
\newblock {\em Arnold Math. J.}, 4(1):19--25, 2018.

\bibitem{Karpenkov2019}
Oleg Karpenkov.
\newblock Geometric conditions of rigidity in nongeneric settings.
\newblock In {\em Handbook of geometric constraint systems principles},
  Discrete Math. Appl. (Boca Raton), pages 317--339. CRC Press, Boca Raton, FL,
  2019.

\bibitem{Karpenkov2013}
Oleg Karpenkov, Jan Schepers, and Brigitte Servatius.
\newblock On stratifications for planar tensegrities with a small number of
  vertices.
\newblock {\em Ars Math. Contemp.}, 6(2):305--322, 2013.

\bibitem{Kitson2014}
Derek Kitson and Stephen~C. Power.
\newblock Infinitesimal rigidity for non-{E}uclidean bar-joint frameworks.
\newblock {\em Bull. Lond. Math. Soc.}, 46(4):685--697, 2014.

\bibitem{Kitson2015}
Derek Kitson and Bernd Schulze.
\newblock Maxwell-{L}aman counts for bar-joint frameworks in normed spaces.
\newblock {\em Linear Algebra Appl.}, 481:313--329, 2015.

\bibitem{Li2008}
Hongbo Li.
\newblock {\em Invariant algebras and geometric reasoning}.
\newblock World Scientific Publishing Co. Pte. Ltd., Hackensack, NJ, 2008.
\newblock With a foreword by David Hestenes.

\bibitem{Maxwell1864}
James~Clerk Maxwell.
\newblock On reciprocal figures and diagrams of forces.
\newblock {\em Philos. Mag.}, 4(27):250--261, 1864.

\bibitem{richter-gebert-2011}
J\"{u}rgen Richter-Gebert.
\newblock {\em Perspectives on projective geometry}.
\newblock Springer, Heidelberg, 2011.

\bibitem{Roth1981}
Ben Roth and Walter Whiteley.
\newblock Tensegrity frameworks.
\newblock {\em Trans. Amer. Math. Soc.}, 265(2):419--446, 1981.

\bibitem{Saliola2007}
Franco~V. Saliola and Walter Whiteley.
\newblock Some notes on the equivalence of first-order rigidity in various
  geometries.
\newblock {\em arXiv:0709.3354 [math.MG]}, pages 1--15, 2007.

\bibitem{Simona-Mariana2011}
Cretu Simona-Mariana and Brinzan Gabriela-Catalina.
\newblock Tensegrity applied to modelling the motion of viruses.
\newblock {\em Acta Mech. Sin.}, 27(1):125--129, 2011.

\bibitem{Sitharam2019}
Meera Sitharam, Audrey St.~John, and Jessica Sidman, editors.
\newblock {\em Handbook of geometric constraint systems principles}.
\newblock Discrete Mathematics and its Applications (Boca Raton). CRC Press,
  Boca Raton, FL, 2019.

\bibitem{Skelton1997}
Robert~E. Skelton.
\newblock Deployable tendon-controlled structure, 1997.
\newblock United States Patent 5642590.

\bibitem{sturmfels-1993}
Bernd Sturmfels.
\newblock {\em Algorithms in invariant theory}.
\newblock Texts and Monographs in Symbolic Computation. Springer-Verlag,
  Vienna, 1993.

\bibitem{wardetzky+2007}
Max Wardetzky, Saurabh Mathur, Felix Kaelberer, and Eitan Grinspun.
\newblock {Discrete Laplace operators: No free lunch}.
\newblock In Alexander Belyaev and Michael Garland, editors, {\em Geometry
  Processing}. The Eurographics Association, 2007.

\bibitem{White1975}
Neil~L. White.
\newblock The bracket ring of a combinatorial geometry. {I}.
\newblock {\em Trans. Amer. Math. Soc.}, 202:79--95, 1975.

\bibitem{white-1995}
Neil~L. White.
\newblock A tutorial on {G}rassmann-{C}ayley algebra.
\newblock In {\em Invariant methods in discrete and computational geometry
  ({C}ura\c{c}ao, 1994)}, pages 93--106. Kluwer Acad. Publ., Dordrecht, 1995.

\bibitem{White1983}
Neil~L. White and Walter Whiteley.
\newblock The algebraic geometry of stresses in frameworks.
\newblock {\em SIAM J. Algebraic Discrete Methods}, 4(4):481--511, 1983.

\bibitem{White1987}
Neil~L. White and Walter Whiteley.
\newblock The algebraic geometry of motions of bar-and-body frameworks.
\newblock {\em SIAM J. Algebraic Discrete Methods}, 8(1):1--32, 1987.

\bibitem{Whiteley1997}
Walter Whiteley.
\newblock Rigidity and scene analysis.
\newblock In {\em Handbook of discrete and computational geometry}, CRC Press
  Ser. Discrete Math. Appl., pages 893--916. CRC, Boca Raton, FL, 1997.

\end{thebibliography}

\end{document}